\documentclass[reqno]{amsart}
\usepackage{amsmath}
\usepackage{amsfonts,amssymb}
\usepackage[dvips]{graphicx}
\usepackage{epsfig}
\usepackage{color}
\newtheorem{theorem}{Theorem}[section]
\newtheorem{remark}{Remark}[section]
\newtheorem{lemma}[theorem]{Lemma}
\newtheorem{definition}{Definition}[section]

\newtheorem{proposition}[theorem]{Proposition}

\numberwithin{equation}{section}
\begin{document}

\title[the compressible Euler equations with a time periodic outer force]
	{Existence of a time periodic solution for the compressible Euler equation with a time periodic outer force in a bounded interval}
\author{Naoki Tsuge}
\address{Department of Mathematics Education, 
Faculty of Education, Gifu University, 1-1 Yanagido, Gifu
Gifu 501-1193 Japan.}
\email{tuge@gifu-u.ac.jp}
\thanks{
N. Tsuge's research is partially supported by Grant-in-Aid for Scientific 
Research (C) 17K05315, Japan.
}
\date{}

\maketitle

\begin{abstract}In the field of differential equations, particularly fluid dynamics, many researchers have shown an interest in the behavior of
time periodic solutions. In this paper, we study isentropic gas flow in a bounded interval and apply a time periodic outer force. This motion is described by the compressible Euler equation with the outer force. Our purpose in this paper is to prove the existence of a time periodic solution. Unfortunately, little is known for the system of conservation laws until now. 
The problem seems to lie in fact that the equation does not possesses appropriate decay estimates.

When we prove the existence of the time periodic solution, we are faced with two difficult problems. 
One problem is to prove that initial data and the corresponding solutions 
at the time period 
are contained in the same bounded set. To overcome this, we employ an invariant region deduced from 
the mass and energy. This enable us to investigate the 
behavior of solutions in detail. In addition, this method provide us a  decay estimate to suppresses 
the growth of solutions caused by the outer force and discontinuities. 
Moreover, there is a possibility that this estimate will lead us to the asymptotic stability for large data in the future. 
Second problem is to construct a continuous map from
initial data to the corresponding solutions at the time period.
We need the map to apply a fixed point theorem. 
To construct this, we introduce 
a new type Lax-Friedrichs scheme, which has a recurrence relation consisting of 
discretized approximate solutions. In virtue of the fixed point theorem, we can prove a existence of a fixed point, which represents a time periodic solution.
Furthermore, the ideas and techniques developed in this paper will be applicable to not only conservation laws but also other nonlinear problems involving similar difficulties such as nonlinear wave equations, the numerical analysis.

Finally, we use the compensated compactness framework to prove the convergence of our approximate solutions. 
\end{abstract}


\section{Introduction}
There has been a great discussion about time periodic solutions in fluid dynamics. 
However, the compressible Euler equation has been little investigated.
The present paper is thus concerned with isentropic gas dynamics
with an outer force. 
\begin{align}
\begin{cases}
\displaystyle{\rho_t+m_x=0,}\\
\displaystyle{m_t+\left(\frac{m^2}{\rho}+p(\rho)\right)_x
	=F(x,t)\rho,}
\end{cases}x\in(0,1),\quad t\in(0,1)
\label{force}
\end{align}
where $\rho$, $m$ and $p$ are the density, the momentum and the 
pressure of the gas, respectively. If $\rho>0$, 
$v=m/\rho$ represents the velocity of the gas. For a barotropic gas, 
$p(\rho)=\rho^\gamma/\gamma$, where $\gamma\in(1,5/3]$ is the 
adiabatic exponent for usual gases. The given function 
$F\in C^1([0,1]\times[0,1])$ represents a time periodic outer force with the time period $1$, 
i.e., $F(x,0)=F(x,1),\;F_t(x,0)=F_t(x,1)$.

We consider the initial boundary value problem (\ref{force}) 
with the initial and boundary data
\begin{align}  
(\rho,m)|_{t=0}=(\rho_0(x),m_0(x))\quad m|_{x=0}=m|_{x=1}=0.
\label{I.C.}
\end{align}
The above problem \eqref{force}--\eqref{I.C.} can be written in the following form 
\begin{align}\left\{\begin{array}{lll}
u_t+f(u)_x=g(x,t,u),\quad{x}\in(0,1),\quad t\in(0,1),\\
u|_{t=0}=u_0(x),\\
m|_{x=0}=m|_{x=1}=0
\label{IP}
\end{array}\right.
\end{align}
by using  $u={}^t(\rho,m)$, $\displaystyle f(u)={}^t\!\left(m, \frac{m^2}{\rho}+p(\rho)\right)$ and 
$\displaystyle{g(x,t,u)={}^t\!\left(0,F(x,t)\rho\right)}$. 

Let us survey the related mathematical results.
Time periodic solutions are widely studied for other differential equations.  For example, Matsumura and Nishida \cite{MN} investigates those of the compressible Navier Stokes equation. 
On the other hand, as far as conservation laws concerned, it has not been received much attention until now. Takeno \cite{Takeno} studies a single conservation law and proved the existence of a time periodic solution for the space periodic boundary condition. 
The key tool is the decay estimate in Tadmor \cite{Tadmor}. 
It should be noted that we cannot apply the method of \cite{Tadmor} to systems. Greenberg and Rascle \cite{GM} treat with an artificial system of conservation laws by 
the Glimm scheme. Although the existence theorem for isentropic gas dynamics is recently obtained in Tsuge \cite{T7}, the initial and boundary conditions are restrictive.

Our goal in this paper is to prove the existence of a time periodic 
solution under a general case. We are then faced with two difficult problems.  
\begin{itemize}
	\item[(P1)] One is to prove that initial data and the corresponding solutions at a period are contained in the same bounded set.
	\item[(P2)] Second is to construct a continuous map in a finite dimension.
\end{itemize}

To overcome (P1), we need an invariant region. 
\cite{T1}--\cite{T8} develop invariant regions with known functions as 
their lower and upper bounds. However, we cannot 
apply their method to the present problem (see Remark \ref{rem:boundary}). To solve this, we employ an invariant region including 
unknown functions such as the mass and energy.  
In addition, 
this method enables us to deduce a decay estimate (see \eqref{estimate1}--\eqref{estimate2}). 
Owing to this estimate, we can control the growth of solutions caused by the outer force and discontinuities. Furthermore, it has the advantage that it is applicable for large data. 
Therefore, this estimate is expected to be used for the analysis of its 
asymptotic stability in the future.

We next consider (P2). To prove the existence of a time periodic solution, we apply the Brouwer fixed point theorem to the continuous map from
initial data to solutions at one period. To construct this map, we introduce 
a new type Lax-Friedrichs scheme, which has a recurrence relation consisting of discretized approximate solutions. The formula yields 
the continuous map in a finite dimension. In addition, the approximate solutions are different from 
those of \cite{T1}--\cite{T8}. Since the approximate solutions consist of 
unknown functions, we must apply the iteration method for their construction in each cell.

\begin{remark}
If we employ the Glimm scheme, we can obtain the decay of the total variation of solutions, which may solves (P1). However, the random choice method of the scheme prevents us from constructing the continuous map in (P2). In addition, the scheme cannot treat with large data.
\end{remark}

To state our main theorem, we define the Riemann invariants $w,z$, which play important roles
in this paper, as
\begin{definition}
	\begin{align*}
	w:=\frac{m}{\rho}+\frac{\rho^{\theta}}{\theta}=v+\frac{\rho^{\theta}}{\theta},
	\quad{z}:=\frac{m}{\rho}-\frac{\rho^{\theta}}{\theta}
	=v-\frac{\rho^{\theta}}{\theta}\quad
	\left(\theta=\frac{\gamma-1}{2}\right).
	\end{align*}
\end{definition}
These Riemann invariants satisfy the following.
\begin{remark}\label{rem:Riemann-invariant}
	\normalfont
	\begin{align*}
	&|w|\geqq|z|,\;w\geqq0,\;\mbox{\rm when}\;v\geqq0.\quad
	|w|\leqq|z|,\;z\leqq0,\;\mbox{\rm when}\;v\leqq0.
    \\
	&v=\frac{w+z}2,
	\;\rho=\left(\frac{\theta(w-z)}2\right)^{1/\theta},\;m=\rho v.
	\end{align*}From the above, the lower bound of $z$ and the upper bound of $w$ yield the bound of $\rho$ and $|v|$.
\end{remark}

Moreover, we define the entropy weak solution and an time periodic entropy weak solution.
\begin{definition}
		A measurable function $u(x,t)$ is called an {\it entropy weak solution} of the initial boundary value problem \eqref{IP} if 
	\begin{align*}
		&\int^{1}_{0}\int^{1}_0u\varphi_t+f(u)\varphi_x+g(x,t,u)\varphi dxdt+\int^{1}_{0}
		u_0(x)\varphi(x,0)dx=0
	\end{align*}
	holds for any test function $\varphi\in C^1([0,1]\times[0,1))$ and 
	\begin{align}
		&\int^{1}_0\int^{1}_0\hspace{-1ex}\eta(u)\psi_t+q(u)\psi_x+\nabla\eta(u) g(x,t,u)\psi dxdt\geqq0
		\label{entropy solution}
	\end{align}
	holds for any non-negative test function $\psi\in C^1((0,1)\times(0,1))$, where 
	$(\eta,q)$ is a pair of convex entropy--entropy flux of \eqref{force}.
\end{definition}	
\begin{definition}

	A measurable function $u(x,t)$ is called an time periodic {\it entropy weak solution} of the initial boundary value problem \eqref{IP} with the period $1$ if 
	there exists a bounded measurable function $u^{\ast}_0(x)	=(\rho^{\ast}_0(x),m^{\ast}_0(x))
	$ such that
\begin{align*}
&\int^{1}_{0}\int^{1}_0u\varphi_t+f(u)\varphi_x+g(x,t,u)\varphi dxdt+\int^{1}_{0}
u^{\ast}_0(x)\left
(\varphi(x,0)-\varphi(x,1)\right)dx=0
\end{align*}
holds for any test function $\varphi\in C^1([0,1]\times[0,1])$ and \eqref{entropy solution}.
\end{definition}

We set 
$\displaystyle \bar{\rho}=\int^1_0\rho_0(x)dx
,\;\bar{\eta}_{\ast}=\int^1_0\eta_{\ast}(u_0(x))dx
$, where $\displaystyle 
		\eta_{\ast}(u)=\frac12\frac{m^2}{\rho}+\frac1{\gamma(\gamma-1)}\rho^{\gamma}.$
Since $\displaystyle \bar{\rho}=0$ implies that the solution becomes vacuum, we assume $\displaystyle \bar{\rho}>0$.

Our main theorems are as follows.
\begin{theorem}\label{thm:main}
For any positive constants $\bar{\rho},\bar{\eta}_{\ast},\varepsilon$ satisfying $0<\varepsilon<\frac{2(\gamma-1)}{\gamma+1}$, 
there exist positive constants $M,\kappa,\alpha$ and a positive function 
$M(t)$ such that the following (A) and (B) hold.

(A) $M(0)=M(1)=M$;

(B) If 
\begin{align}
\Vert F\Vert_{L^{\infty}([0,1]\times[0,1])}\leq\kappa
\label{condition2}
\end{align}
and $u_0\in{L}^{\infty}([0,1])$ satisfy
	\begin{equation}
		\begin{split}
			&\rho_0(x)\geq0,\quad -M+\int^x_{0}\zeta(u_0(y))dy
			\leq z(u_0(x)),\\
			&w(u_0(x))\leq M+\int^x_{0}\zeta(u_0(y))dy,
		\end{split}
		\label{initial}
	\end{equation}

		then, there exists a solution $u(x,t)$ of the initial boundary problem \eqref{IP} such that, for $(x,t)\in[0,1]\times[0,1]$, 
\begin{equation}
		\begin{alignedat}{2}
			&\rho(x,t)\geq&&0,\quad  -M(t)+\int^x_{0}\zeta(u(y,t))dy\leq z(x,t),\\
			&w(x,t)\leq&& M(t)+\int^x_{0}\zeta(u(y,t))dy,
		\end{alignedat}
		\label{inequaltiy}
	\end{equation}where \begin{align}
	\begin{alignedat}{2}
		K=\alpha\bar{\rho}-\int^1_0\eta_{\ast}(u_0(x))dx-1=M^{\frac{2(\gamma-1)}{\gamma+1}-\varepsilon}
	\end{alignedat}
	\label{condition1}
\end{align}
and 
\begin{align}
\zeta(u)=\eta_{\ast}(u)-\alpha\rho+K.
\label{zeta}
\end{align}
\end{theorem}	
We deduce from Theorem \ref{thm:main} the following theorem.
\begin{theorem}\label{thm:periodic}
	There exists a time periodic entropy weak solution of the initial boundary value problem \eqref{IP}. 
\end{theorem}
\begin{remark}\label{rem:bound}\normalfont
We will deduce from \eqref{inequaltiy} that 
	\begin{align}
		\begin{alignedat}{1}
			&|z(u(x,t))|=O(M),\;|w(u(x,t))|=O(M),\quad(x,t)\in[0,1]\times[0,1].
		\end{alignedat}
		\label{bound1}
	\end{align}
In addition, it will follows from the conservation of mass and 
energy inequality that 
	\begin{align}
		\begin{alignedat}{1}
			&\int^x_{0}\eta_{\ast}(u(y,t))dy=O(1),\;
			\int^x_{0}\rho(y,t)dy=O(1),
			\;(x,t)\in[0,1]\times[0,1].
		\end{alignedat}
		\label{bound2}
\end{align}
We notice that $O(1)$ is independent of 
$M$.

In view of \eqref{initial} and \eqref{inequaltiy}, we find that our solution are contained in the same bounded 
set.
\end{remark}
\begin{remark}\label{rem:boundary}
We let the lower and upper bounds in \eqref{inequaltiy} be 
	\begin{align*}
		\begin{alignedat}{2}
			&L(x,t;u)=-M(t)+\int^x_{0}\left\{\eta_{\ast}(u(y,t))-\alpha\rho(y,t)+K\right\}dy,\\
			&U(x,t;u)=M(t)+\int^x_{0}\left\{\eta_{\ast}(u(y,t))-\alpha\rho(y,t)+K\right\}dy,
		\end{alignedat}
	\end{align*}
respectively. Then we notice that 
	\begin{align}
		-L(0,t;u)\leq U(0,t;u),\;-L(1,t;u)\geq U(1,t;u).
		\label{boundary-bound}
	\end{align}
In fact, the former is clear. The latter is deduced from \eqref{condition1}, \eqref{condition2} and \eqref{bound1} as follows.
	\begin{align}
		L(1,t;u)+U(1,t;u)=&2\int^1_{0}\left\{\eta_{\ast}(u(x,t))-\alpha\rho(x,t)+K\right\}dx\nonumber\\
		=&2\int^1_{0}\left\{\left(\eta_{\ast}(u(x,t))-\eta_{\ast}({u}_0(x))\right)
		-\alpha\left(\rho(x,t)-\bar{\rho}\right)-1\right\}dx\nonumber\\
		\leq&2\int^1_0\int^t_0F(x,s)m(x,s)dxds-1\nonumber\\
		\leq&0,
	\end{align}
choosing $\kappa$ small enough.

\eqref{boundary-bound} is a necessary condition that 
\eqref{inequaltiy} is an invariant region with boundary data $m=0$.

Tsuge \cite{T1}--\cite{T8} propose various invariant regions. 
Their lower and upper bounds consist of known functions, which are 
increasing. The property plays an important role for their analysis. However, 
they cannot satisfy \eqref{boundary-bound}. To solve this, we introduce 
an invariant region consisting of not known functions but unknown functions 
such as the mass and energy (see \eqref{inequaltiy}).
\end{remark}

\subsection{Outline of the proof (formal argument)}

The proof of main theorem is a little complicated. Therefore, 
before proceeding to the subject, let us grasp the point of the main estimate by a formal argument. 
We assume that a solution is smooth and the density is nonnegative in this section.

We consider the physical region $\rho\geqq0$ (i.e., $w\geqq z$.). Recalling Remark \ref{rem:Riemann-invariant}, it suffices to 
derive the lower bound of $z(u)$ and the upper bound of $w(u)$ to obtain the bound of $u$. To do this, we diagonalize \eqref{force}. 
If solutions are smooth, we deduce from \eqref{force} 
\begin{align}
z_t+\lambda_1z_x=F(x,t),\quad
w_t+\lambda_2w_x=F(x,t),
\label{force2}
\end{align} 
where $\lambda_1$ and $\lambda_2$ are the characteristic speeds defined as follows 
\begin{align}
\lambda_1=v-\rho^{\theta},\quad\lambda_2=v+\rho^{\theta}.
\label{char}
\end{align}

We introduce $\tilde{z},\tilde{w}$ as follows.
\begin{align}
	\begin{split}
		z=\tilde{z}+\int^x_{0}\left\{\eta_{\ast}(u)-\alpha\rho+K\right\}dy,\quad w=\tilde{w}+\int^x_{0}\left\{\eta_{\ast}(u)-\alpha\rho+K\right\}dy.
	\end{split}
	\label{transformation}
\end{align}
We deduce from the conservation of mass and energy that
\begin{align}
	\tilde{z}_t+\lambda_1 \tilde{z}_x=g_1(x,t,u),\quad
	\tilde{w}_t+\lambda_2 \tilde{w}_x=g_2(x,t,u),
	\label{Riemann2}
\end{align}
where 
\begin{align}
	\begin{alignedat}{2}
		&g_1(x,t,u)=&&-K\lambda_1+\dfrac{1}{\gamma(\gamma-1)}\rho^{\gamma+\theta}
		+\dfrac{1}{\gamma}\rho^{\gamma}v+\dfrac{1}{2}\rho^{\theta+1}v^2-\alpha\rho^{\theta+1}\\
		&&&+F(x,t)-\int^x_0F(y,t)m(y,t)dy,\\
		&g_2(x,t,u)=&&-K\lambda_2-\dfrac{1}{\gamma(\gamma-1)}\rho^{\gamma+\theta}
		+\dfrac{1}{\gamma}\rho^{\gamma}v-\dfrac{1}{2}\rho^{\theta+1}v^2+\alpha\rho^{\theta+1}\\
		&&&+F(x,t)-\int^x_0F(y,t)m(y,t)dy.
	\end{alignedat}
\label{inhomo2}
\end{align}
Then, we notice that
\begin{align*}
	-M\leq \tilde{z}_0(x),\quad \tilde{w}_0(x)\leq M,\quad \tilde{\rho}_0(x)\geq0.
\end{align*}

Let us prove that 
\begin{align*}
	S_{inv}=\{(\tilde{z},\tilde{w})\in{\bf R}^2;\tilde{\rho}\geq0,\;\tilde{z}\geq-M,\;\tilde{w}\leq M\}
\end{align*}
is an invariant region. 

To achieve this, assuming that 
\begin{align*}
	-M< \tilde{z}_0(x),\quad \tilde{w}_0(x)< M
\end{align*}
and there exist $x_{\ast}\in(0,1),\;0<t_{\ast}\leq1$ such that 
\eqref{invariant1} or \eqref{invariant2} holds, we 
deduce a contradiction, where
\begin{align}
\begin{alignedat}{2}
	&-M<\tilde{z}(x,t),\;\tilde{w}(x,t)< M,\quad x\in(0,1),\;
	0\leq t<t_{\ast}\\&\text{\hspace*{0ex}and}\quad
	\tilde{z}(x_{\ast},t_{\ast})=-M,\;\tilde{w}(x_{\ast},t_{\ast})\leq M,	
\end{alignedat}\label{invariant1}\\	
\begin{alignedat}{2}
	&-M<\tilde{z}(x,t),\;\tilde{w}(x,t)< M,\quad x\in(0,1),\;
	0\leq t<t_{\ast}\\&\text{\hspace*{0ex}and}\quad
	\tilde{z}(x_{\ast},t_{\ast})\geq-M,\;\tilde{w}(x_{\ast},t_{\ast})=M.
\end{alignedat}\label{invariant2}
\end{align}

To do this, we prove 
\begin{align}	&g_1(x_{\ast},t_{\ast},u)>0\text{, when \eqref{invariant1} holds},
	\label{g1}\\
	&g_2(x_{\ast},t_{\ast},u)<0\text{, when \eqref{invariant2} holds}.
	\label{g2}
\end{align} 
We first investigate \eqref{g2}.

Under \eqref{invariant2}, for $0\leq t \leq t_{\ast}$, we 
show that the energy is bounded.
From \eqref{Riemann2}, we notice $\tilde{\rho}=\rho$. Observing 
\begin{align*}
	\dfrac{(\rho(x,t))^{\theta}}{\theta}=\dfrac{({\tilde{\rho}(x,t)})^{\theta}}{\theta}=\dfrac{\tilde{w}(x,t)-\tilde{z}(x,t)}2\leq M,
\end{align*}we have 
\begin{align}
	\rho(x,t)\leq (\theta M)^{\frac1{\theta}}
	\label{1}
\end{align}
and 
\begin{align}
	\begin{alignedat}{2}
		|v(x,t)|=&\left|\tilde{v}(x,t)+\int^x_0\zeta(u(y,t))dy\right|\leq |\tilde{v}(x,t)|+\left|\int^x_0\zeta(u(y,t))dy\right|\\
		\leq& M
		+\int^1_0\eta_{\ast}(u(x,t))dx+\alpha\bar{\rho}+K.
	\end{alignedat}
	\label{2}
\end{align}
From \eqref{condition2}, \eqref{1}, \eqref{2} and the energy inequality, we obtain
\begin{align}
	\begin{alignedat}{2}
		\int^1_0&\eta_{\ast}(u(x,t))dx\leq\int^1_0\eta_{\ast}(u_0(x))dx+\int^t_0\int^1_0m(x,s)F(x,s)dxds\\
		\leq& \bar{\eta}_{\ast}+\int^t_0 \kappa(\theta M)^{\frac1{\theta}}
		\left(M+\int^1_0\eta_{\ast}(u(x,s))dx+\alpha\bar{\rho}+K\right)ds\\
		\leq& C+\int^t_0 \kappa(\theta M)^{\frac1{\theta}}\int^1_0\eta_{\ast}(u(x,s))dxds,
	\end{alignedat}
\end{align}
where $C=\bar{\eta}_{\ast}+\kappa(\theta M)^{\frac1{\theta}}
\left(M+\alpha\bar{\rho}+K\right)$. We deduce from the Gronwall inequality 
\begin{align*}
	\int^1_0&\eta_{\ast}(u(x,t))dx\leq C\exp\left|\int^t_0 \kappa(\theta M)^{\frac1{\theta}}ds\right|
	\leq C\exp(\kappa(\theta M)^{\frac1{\theta}}).
\end{align*} Choosing $\kappa$ small enough, we have 
\begin{align}
	\int^1_0\eta_{\ast}(u(x,t))dx=O(1)\quad 0\leq t\leq t_{\ast},
	\label{bound energy}
\end{align}
where $O(1)$ is the Landau symbol as $M\rightarrow\infty$.

In this case, from \eqref{condition1} and \eqref{bound energy}, we have
\begin{align}
	\begin{alignedat}{2}
		&K=M^{\frac{2(\gamma-1)}{\gamma+1}-\varepsilon},\quad \alpha=M^{\frac{2(\gamma-1)}{\gamma+1}-\varepsilon}/\bar{\rho}+O(1),
		\quad \rho=\tilde{\rho},\quad \tilde{v}\geq0,\\
		&\lambda_2=\tilde{\lambda}_2+O(M^{\frac{2(\gamma-1)}{\gamma+1}-\varepsilon})
		, \quad\tilde{\lambda}_2=\tilde{v}+\tilde{\rho}^{\theta}
		=M-\left(1/\theta-1\right)\rho^{\theta}\geq\theta M,
	\end{alignedat}
\end{align}
where $O(1)$ and $O(M^{\frac{2(\gamma-1)}{\gamma+1}-\varepsilon})$ are the Landau symbol as $M\rightarrow\infty$.

Separating two parts, we shall prove \eqref{g2}.
\begin{enumerate}
	\item $\rho>\left(\dfrac{\bar{\rho}M}{3}\right)^{\frac{1}{\theta+1}}$

	For $(x,t)=(x_{\ast},t_{\ast})$, from $\gamma>1$, $\eqref{condition1}$ and $\eqref{condition2}$, we have
	\begin{align}
		\begin{alignedat}{3}
			&g_2(x,t,u)&&\leq&&-K\tilde{\lambda}_2
			-\dfrac{\gamma+1}{2\gamma^2(\gamma-1)}\rho^{\gamma+\theta}
-\dfrac{\rho^{\theta+1}}2\left({v}-\frac{\rho^{\theta}}{\gamma}\right)^2
			+\alpha\rho^{\theta+1}\\&&&&&+O(M^{\frac{4(\gamma-1)}{\gamma+1}-2\varepsilon})\\
			&&&\leq&&-\rho^{\theta+1}\left\{\dfrac{\gamma+1}{2\gamma^2(\gamma-1)}\rho^{\gamma-1}-\alpha \right\}+O(M^{\frac{4(\gamma-1)}{\gamma+1}-2\varepsilon})\\
			&&&\leq&&
			-\frac12M
			^{1+\frac{2(\gamma-1)}{\gamma+1}-\varepsilon}
,
			\end{alignedat}
		\label{estimate1}
	\end{align}choosing $M$ large enough and $\kappa$ small enough.
	
	\item $\rho\leq \left(\dfrac{\bar{\rho}M}{3}\right)^{\frac{1}{\theta+1}}$

	For $(x,t)=(x_{\ast},t_{\ast})$, we have 
	\begin{align}
		\begin{alignedat}{2}
			&g_2(x,t,u)&\leq&-K\tilde{\lambda}_2
			-\dfrac{\gamma+1}{2\gamma^2(\gamma-1)}\rho^{\gamma+\theta}
			-\dfrac{\rho^{\theta+1}}2\left({v}-\frac{\rho^{\theta}}{\gamma}\right)^2
			+\alpha\rho^{\theta+1}\\&&&+O(M^{\frac{4(\gamma-1)}{\gamma+1}-2\varepsilon})\\
			&&\leq&-M^{\frac{2(\gamma-1)}{\gamma+1}-\varepsilon}\left(M-(1/\theta-1)\left(\dfrac{\bar{\rho}M}{2}\right)^{\frac{\theta}{\theta+1}} \right)+\dfrac{M}{3}M^{\frac{2(\gamma-1)}{\gamma+1}-\varepsilon}\\&&&+O(M^{\frac{4(\gamma-1)}{\gamma+1}-2\varepsilon})\\&&<&-\frac12M
			^{1+\frac{2(\gamma-1)}{\gamma+1}-\varepsilon}
,
		\end{alignedat}
	\label{estimate2}
	\end{align}choosing $M$ large enough and $\kappa$ small enough.
\end{enumerate}Therefor, we complete the proof of \eqref{g2}. On the other hand, 
since $\tilde{w}$ attains the maximum at $(x,t)=(x_{\ast},t_{\ast})$, we find that 
$\tilde{w}_t(x_{\ast},t_{\ast})\geq0,\;\tilde{w}_x(x_{\ast},t_{\ast})=0$. Then, from $\eqref{Riemann2}_2$, 
we have $g_2(x,t,u)\geq0$ at $(x,t)=(x_{\ast},t_{\ast})$. This is a contradiction.
We can similarly prove \eqref{g1} and deduce a contradiction.

It should be noted that \eqref{estimate1}--\eqref{estimate2} 
yields a decay estimate.

\begin{remark}
We review the role of each component of $\zeta(u)$ in the above argument. 
We recall that $\zeta(u)$  in \eqref{zeta} consits of tree terms $\eta_{\ast}(u),\;
\alpha\rho$ and $K$. 
When the 
density is large (i), $\eta_{\ast}(u)$ is a leading term in 
 \eqref{estimate1}. On the other hand, 
when the density is small (ii), so is $K$ in \eqref{estimate2}. 
However, if $\zeta(u)$ has only these two terms, \eqref{boundary-bound} does not 
hold. To solve this, we add $\alpha \rho$ to $\eta_{\ast}(u)$. These terms thus play the role of trinity.

Since \eqref{force} has a discontinuous solution, the above argument is formal. 
In fact, $S_{inv}$ is not an invariant region for our problem (see \eqref{bound2}) exactly, 
because our weak solutions increase due to their discontinuities, whose 
quantity is denoted by $J^n_j$ in \eqref{functional discontinuity}. We will treat with $J^n_j$ 
by the decay estimate \eqref{estimate1}--\eqref{estimate2}. 
\end{remark}

Next, we prove the existence of a time periodic solution.
We find that both $(\tilde{z}_0(x),\tilde{w}_0(x))$ and $(\tilde{z}(x,1),\tilde{w}(x,1))$
are containded in $S_{inv}$. Therfore, applying the fixed point theorem, we obtain a fixed point $(\tilde{z}^*_0(x),\tilde{w}^*_0(x))=(\tilde{z}^*(x,1),\tilde{w}^*(x,1))$. (Exactly speaking, 
we apply the Brauwer fixed point theorem to a sequence deduced from a difference scheme.)
This implies $(\tilde{\rho}^*_0(x),\tilde{v}^*_0(x))=(\tilde{\rho}^*(x,1),\tilde{v}^*(x,1))$. 
However, we must prove a fixed point for original unknown functions.

First, since $\rho=\tilde{\rho}$, we have ${\rho}^*_0(x)={\rho}^*(x,1)\; x\in[0,1]$.
Next, let us prove ${v}^*_0(x)={v}^*(x,1)\; x\in[0,1]$.
Recalling \eqref{transformation}, we find that
\begin{align}
	&v^*(x,1)=\tilde{v}^*(x,1)+\int^x_{0}\left\{\eta_{\ast}(u^*)-\alpha\rho^*+K\right\}dy,\\
	&v^*_0(x)=\tilde{v}^*_0(x)+\int^x_{0}\left\{\eta_{\ast}(u^*_0)-\alpha\rho^*_0+K\right\}dy.
\end{align}
From ${\rho}^*_0(x)={\rho}^*(x,1)$, we obtain
\begin{align}
	v^*(x,1)-v^*_0(x)=\int^x_{0}\frac12\rho^*_0(x)(v^*(y,1)+v^*_0(y))(v^*(y,1)-v^*_0(y))dy.
	\label{fixed}
\end{align}
We assume that there exists a point $x^{\star}\in(0,1)$ such that $v^*_0(x^{\star})\ne
v^*(x^{\star},1)$. Then, we set $\displaystyle x^{\star}_0=\inf_x
\left\{x\in[0,1];x<x^{\star},\;v^*_0(x)\ne v^*(x,1)\right\}$.
From \eqref{fixed}, since $v^*(0,1)-v^*_0(0)=0$, we find that
$v^*(x^{\star}_0,1)-v^*_0(x^{\star}_0)=0$. Differentiating
\eqref{fixed}, deviding the resultant equation by $v^*(x,1)-v^*_0(x)$ and integrating the resultant one from 
$x^{\star}_0$ to $x^{\star}$, we have 
\begin{align}
	\begin{alignedat}{2}
		&\log|v^*(x^{\star},1)-v^*_0(x^{\star})|-\log|v^*(x^{\star}_0,1)-v^*_0(x^{\star}_0)|\\
		&=\int^{x^{\star}}_{x^{\star}_0}\frac12\rho^*_0(x)(v^*(y,1)+v^*_0(y))dy.
	\end{alignedat}
\end{align}
$\log|v^*(x^{\star}_0,1)-v^*_0(x^{\star}_0)|$ is $-\infty$. On the other hand, 
the right hand side is bounded. This is a contradiction.

Although the above argument is formal, it is essential. In fact, we shall implicitly use this property in Section 3--4. However, we cannot justify 
the above argument by the standard difference scheme such as Godunov or Lax-Friedrichs 
scheme. Therefore, we introduce a new type Lax Friedrichs scheme in Section 2. Recently, the various difference schemes are developed in \cite{T1}--\cite{T8}, 
which consist of known functions. On the other hand, the present approximate solutions 
include unknown functions in the form of \eqref{transformation} with constants $\tilde{z},\tilde{w}$ (see \eqref{appro1}).

The present paper is organized as follows.
In Section 2, we construct approximate solutions by 
the Lax Friedrichs scheme mentioned above. In Section 3, we drive the bounded estimate of our approximate solutions. In Section 4, we prove the existence of a fixed point by using a recurrence relation which is deduced from our approximate solutions.

\section{Construction of Approximate Solutions}
\label{sec:construction-approximate-solutions}
In this section, we construct approximate solutions. In the strip 
$0\leqq{t}\leqq{1}$, we denote these 
approximate solutions by $u^{\varDelta}(x,t)
=(\rho^{\varDelta}(x,t),m^{\varDelta}(x,t))$. 
For $N_x\in{\bf N}$, we define the space mesh
lengths by ${\varDelta}x=1/(2N_x)$. 
Using $M$ in \eqref{condition2}, we take time mesh length ${\varDelta}{x}$ such that 
\begin{align}
	\frac{{\varDelta}x}{{\varDelta}{t}}=[\hspace{-1.2pt}[2M]\hspace{-1pt}]+1,
\label{CFL}
\end{align}
where $[\hspace{-1.2pt}[x]\hspace{-1pt}]$ is the greatest integer 
not greater than $x$. Then we define $N_t=1/(2{\varDelta t})\in{\bf N}$.
In addition, 
we set 
\begin{align*}
	(j,n)\in{\bf N}_x\times{\bf N}_t,
\end{align*}
where ${\bf N}_x=\{0,1,2,\ldots,2N_x\}$ and ${\bf N}_t=\{0,1,2,\ldots,2N_t\}$.  For simplicity, we use the following terminology
\begin{align}
\begin{alignedat}{2}
&x_j=j{\varDelta}x,\;t_n=n{\varDelta}t,\;t_{n.5}=\left(n+\frac12\right){\varDelta}t,
\;t_{n-}=n{\varDelta}t-0,\;t_{n+}=n{\varDelta}t+0.
\end{alignedat}
\label{terminology} 
\end{align}

First we define $u^{\varDelta}(x,-0)$ by $u^{\varDelta}(x,-0)=u_0(x)$
and set 
\begin{align*}
	J_n=\{k\in{\bf N}_x;k+n=\text{odd}\}.
\end{align*}
Then, for $j\in J_0$, we define $E_j^0(u)$ by
\begin{align*}
	E_j^0(u)=\frac1{2{\varDelta}x}\int^{x_{j+1}}_{x_{j-1}}
	u^{\varDelta}(x,-0)dx.
\end{align*}

Next, assume that $u^{\varDelta}(x,t)$ is defined for $t<{t}_{n}$. 

\vspace*{1ex}

(i) $n$ is even

\vspace*{1ex}

Then, for $j\in J_n$, we define $E^n_j(u)$ by 
\begin{align*}
	E^n_j(u)=\frac1{2{\varDelta}x}\int_{x_{j-1}}^{x_{j+1}}u^{\varDelta}(x,t_{n-})dx.
\end{align*}

\vspace*{1ex}

(ii) $n$ is odd

\vspace*{1ex}

Then, for $j\in J_n\setminus\{0,2N_x\}$, we define $E^n_j(u)$ by 
\begin{align*}
	E^n_j(u)=\frac1{2{\varDelta}x}\int_{x_{j-1}}^{x_{j+1}}u^{\varDelta}(x,t_{n-})dx;
\end{align*}
for $j\in \{0,2N_x\}$, we define $E^n_j(u)$ by 
\begin{align*}
	E^n_0(u)=&\frac1{{\varDelta}x}\int_{0}^{x_1}u^{\varDelta}(x,t_{n-})dx,\;
	E^n_{2N_x}(u)=\frac1{{\varDelta}x}\int_{x_{2N_x-1}}^
{x_{2N_x}}u^{\varDelta}(x,t_{n-})dx.
\end{align*}
Let $E^n(x;u)$ be a piecewise constant function defined by
\begin{align*}
	E^n(x;u)=
	\begin{cases}
		E^n_j(u),\quad &x\in [x_{j-1},x_{j+1})\quad (j\in J_n,\;\text{$n$ is even})	,\vspace*{0.5ex}\\
		E^n_j(u),\quad &x\in [x_{j-1},x_{j+1})\quad (j\in J_n,\;j\ne0,2N_x,\;\;\text{$n$ is odd})	,\vspace*{0.5ex}\\
		E^n_0(u),\quad &x\in [0,x_1)\quad (j=0,\;\;\text{$n$ is odd}),\vspace*{0.5ex}\\
	E^n_{2N_x}(u),&x\in [x_{2N_x-1},x_{2N_x})\quad (j=2N_x,\;\;\text{$n$ is odd}).	
	\end{cases}
\end{align*}

To define $u_j^n=(\rho_j^n,m_j^n)$ for $j\in J_n$, we first define symbols $I^n_j$ and $L^n_j$. Let the approximation of $\zeta(u)$ be
\begin{align*}
	I^n_j:=
	\begin{cases}\displaystyle 
	\int^{x_{j-1}}_{x_0}\zeta(E^n(x;u))dx 
		+\frac{1}2\int_{x_{j-1}}^{x_{j+1}}\zeta(E^n(x;u))dx&\\\displaystyle \quad
         =\int^{x_{j}}_{x_0}\zeta(E^n(x;u))dx,\;&\text{$n$ is even,}\vspace*{1ex}\\
		\displaystyle \int^{x_{j-1}}_{0}\zeta(E^n(x;u))dx
		+\frac{1}2\int_{x_{j-1}}^{x_{j+1}}\zeta(E^n(x;u))dx&\\\displaystyle \quad
         =\int^{x_{j}}_{x_0}\zeta(E^n(x;u))dx,&\text{$n$ is odd,\;$j\ne 0,2N_x$,}\vspace*{1ex}\\
		\displaystyle \frac{1}2\int^{x_1}_{0}\zeta(E^n(x;u))dx	,\;&\text{$n$ is odd,\;$j=0$,}\vspace*{1ex}\\
		\displaystyle \int^{x_1}_{0}\zeta(E^n(x;u))dx
		+\frac{1}2\int^{x_{2N_x}}_{x_{2N_x-1}}\zeta(E^n(x;u))dx,\;&\text{$n$ is odd,\;	$j=2N_x$,}	
	\end{cases}
\end{align*}where $\zeta$ is defined in \eqref{zeta}.

Let ${\mathcal D}=(x(t),t)$ denote a discontinuity in 
${u}^{\varDelta}(x,t),\;[\eta_{\ast}]$ and $[q_{\ast}]$ 
denote the jump of $\eta_{\ast}({u}^{ \varDelta}(x,t))$ and $q_{\ast}({u}^{ \varDelta}(x,t))$ across ${\mathcal D}$ from 
left to right, respectively,
\begin{align*}
&[\eta_{\ast}]=\eta_{\ast}({u}^{\varDelta}(x(t)+0,t))-\eta_{\ast}({u}^{ \varDelta}(x(t)-0,t)),
\\&
{[q_{\ast}]=q_{\ast}({u}^{ \varDelta}(x(t)+0,t))-q_{\ast}({u}^{ \varDelta}(x(t)-0,t))}.
\end{align*}
To measure the error in the entropy condition and the gap of the 
energy at $t_{n\pm}$, we introduce the following functional.
\begin{align}
	\begin{alignedat}{2}
	L^n_j=&\int^{t_{n}}_{0}\sum_{0\leq x\leq 1}\sigma[\eta_{\ast}]-[q_{\ast}]dt
	+\sum_{n\in{\bf N}_t}\int^1_0\left\{\eta_{\ast}({u}^{ \varDelta}(x,t_{n-0}))-\eta_{\ast}(E^n(x;u))\right\}dx\\&
	+\left(1+C_{\gamma}\alpha\int^1_0\rho_0(x)dx\right)\sum_{\substack{j\in J_n\\n\in {\bf N}_t}}\frac1{2{\varDelta}x}\int^{x_{j+1}}_{x_{j-1}}(x_{j+1}-x)R^n_j(x)dx,
	\end{alignedat}
\label{functional discontinuity}	
\end{align}
where
\begin{align}
C_{\gamma}=\max\left\{2^{\theta}(\theta+1),\dfrac{2\gamma(\gamma-1)}{\gamma-2+\left(\frac12\right)^{\gamma-1}}\right\},
\label{CGamma}
\end{align} 
\begin{align*}
R^n_j(x)=&\int^1_0(1-\tau)\cdot{}^t\left({u}^{\varDelta}(x,t_{n-})-u^n_j\right)
\nabla^2\eta_{\ast}\left(u^n_j+\tau\left\{{u}^{\varDelta}(x,t_{n-})-u^n_j\right\}\right)\\&\times\left({u}^{\varDelta}(x,t_{n-})-u^n_j\right)d\tau
\end{align*}
and the summention in $\sum_{0\leq x\leq1}$
is taken over all discontinuities in ${u}^{ \varDelta}(x,t)$ at a fixed time $t$ over 
$x\in[0,1]$,
$\sigma$ is the propagating speed of the discontinuities.

From the entropy condition, $\sigma[\eta_{\ast}]-[q_{\ast}]\geq0$. From the Jensen inequality, 
$\int^1_0\left\{\eta_{\ast}({u}^{ \varDelta}(x,t_{n-0}))-\eta_{\ast}(E^n(x;u))\right\}dx\geq0$. Therefore, we find that $L^n_j\geq0$.

Using $I^n_j,L^n_j$, we define $u_j^n$ as follows.

We choose $\delta$ such that $1<\delta<1/(2\theta)$. If 
\begin{align*}
	E^n_j(\rho):=
	\frac1{2{\varDelta}x}\int_{x_{j-1}}^{x_{j+1}}\rho^{\varDelta}(x,t_{n-})dx<({\varDelta}x)^{\delta},
\end{align*} 
we define $u_j^n$ by $u_j^n=(0,0)$;
otherwise, setting\begin{align}
		{z}_j^n:=
		\max\left\{z(E_j^n(u)),\;-M_n-L^n_j+I^n_j\right\},\;
		w_j^n:=\min\left\{w(E_j^n(u)),\;M_n+L^n_j+I^n_j\right\}
		,
	\label{def-u^n_j}
\end{align}

we define $u_j^n$ by
\begin{align*}
	u_j^n:=(\rho_j^n,m_j^n):=(\rho_j^n,\rho_j^nv^n_j)
	:=\left(\left\{\frac{\theta(w_j^n-z_j^n)}{2}\right\}
	^{1/\theta},
	\left\{\frac{\theta(w_j^n-z_j^n)}{2}\right\}^{1/\theta}
	\frac{w_j^n+z_j^n}{2}\right).
\end{align*}

\begin{remark}\label{rem:E}\normalfont
We find 
	\begin{align}
		\begin{split}
			-M_n-L^n_j+I^n_j\leqq z(u_j^n),\quad
			{w}(u_j^n)\leqq M_n+L^n_j+I^n_j.
		\end{split}
		\label{remark2.1}
	\end{align}

	This implies that we cut off the parts where 
	$z(E_j^n(u))<-M_n-L^n_j+I^n_j$
	and $w(E_j^n(u))>M_n+L^n_j+I^n_j$
	in  defining $z(u_j^n)$ and 
	${w}(u_j^n)$. Observing \eqref{average}, the order of these cut parts is $o({\varDelta}x)$. The order is so small that we can deduce the compactness and convergence of our approximate solutions.

\end{remark}

We must construct our approximate solutions $u^{\varDelta}(x,t)$ near 
the boundary and in an interior domain. The construction 
of two cases is similar. Therefore, we are devoted to treating with the construction in 
the cell in the interior domain.

\subsection{Construction of Approximate Solutions in the Cell of the interior domain}
\label{subsec:construction-approximate-solutions}
We then assume that 
approximate solutions $u^{\varDelta}(x,t)$ are defined in domains $D_1: 
t<{t}_n\quad(
n\in {\bf N}_t)$ and $D_2:
x<x_{j-1}\quad(j\in J_{n+1}),\;{t}_n\leqq{t}<t_{n+1}$. 
By using $u_j^n$ defined above and $u^{\varDelta}(x,t)$ defined in $
D_2$, we 
construct the approximate solutions  in the cell $n{\varDelta}{t}\leqq{t}<(n+1){\varDelta}{t}\quad 
(n\in{\bf N}_t),\quad x_{j-1}\leqq{x}<x_{j+1}\quad
(j\in J_n\setminus\{0,1,2N_x-1,2N_x\})$.

We first solve a Riemann problem with initial data $(u_{j-1}^n,u_{j+1}^n)$. 
Call constants $u_{\rm L}(=u_{j-1}^n), u_{\rm M}, u_{\rm R}(=u_{j+1}^n)$ the left, middle and 
right states, respectively. Then the following four cases occur.
\begin{itemize}
	\item {\bf Case 1} A 1-rarefaction wave and a 2-shock arise. 
	\item {\bf Case 2} A 1-shock and a 2-rarefaction wave arise. 
	\item {\bf Case 3} A 1-rarefaction wave and a 2-rarefaction arise.
	\item {\bf Case 4} A 1-shock and a 2-shock arise.
\end{itemize}
We then construct approximate solutions $u^{\varDelta}(x,t)$ by perturbing 
the above Riemann solutions.

Let $\alpha$ be a constant satisfying $1/2<\alpha<1$. Then we can choose 
a positive value $\beta$ small enough such that $\beta<\alpha$, $1/2+\beta/2<\alpha<
1-2\beta$, $\beta<2/(\gamma+5)$ and $(9-3\gamma)\beta/2<\alpha$.

In this step, we 
consider Case 1 in particular. The constructions of Cases 2--4 are similar 
to that of Case 1. We consider only the case in which $u_{\rm M}$ is away from the vacuum. The other case (i.e., the case where $u_{\rm M}$ is near the vacuum) is a little technical. Therefore, we postpone this case to Appendix B.

Consider the case where a 1-rarefaction wave and a 2-shock arise as a Riemann 
solution with initial data $(u_j^n,u_{j+1}^n)$. Assume that 
$u_{\rm L},u_{\rm M}$ 
and $u_{\rm M},u_{\rm R}$ are connected by a 1-rarefaction and a 2-shock 
curve, respectively. \vspace*{10pt}\\
{\it Step 1}.\\
In order to approximate a 1-rarefaction wave by a piecewise 
constant {\it rarefaction fan}, we introduce the integer  
\begin{align*}
	p:=\max\left\{[\hspace{-1.2pt}[(z_{\rm M}-z_{\rm L})/({\varDelta}x)^{\alpha}]
	\hspace{-1pt}]+1,2\right\},
\end{align*}
where $z_{\rm L}=z(u_{\rm L}),z_{\rm M}=z(u_{\rm M})$ and $[\hspace{-1.2pt}[x]\hspace{-1pt}]$ is the greatest integer 
not greater than $x$. Notice that
\begin{align}
	p=O(({\varDelta}x)^{-\alpha}).
	\label{order-p}
\end{align}
Define \begin{align*}
	z_1^*:=z_{\rm L},\;z_p^*:=z_{\rm M},\;w_i^*:=w_{\rm L}\;(i=1,\ldots,p),
\end{align*}
and 
\begin{align*}
	z_i^*:=z_{\rm L}+(i-1)({\varDelta}x)^{\alpha}\;(i=1,\ldots,p-1).
\end{align*}
We next introduce the rays $x=(j+1/2){\varDelta}x+\lambda_1(z_i^*,z_{i+1}^*,w_{\rm L})
(t-n{\varDelta}{t})$ separating finite constant states 
$(z_i^*,w_i^*)\;(i=1,\ldots,p)$, 
where  
\begin{align*}
	\lambda_1(z_i^*,z_{i+1}^*,w_{\rm L}):=v(z_i^*,w_{\rm L})
	-S(\rho(z_{i+1}^*,w_{\rm L}),\rho(z_i^*,w_{\rm L})),
\end{align*}
\begin{align*}
	\rho_i^*:=\rho(z_i^*,w_{\rm L}):=\left(\frac{\theta(w_{\rm L}-z_i^*)}2\right)^{1/\theta}\;,
	\quad{v}_i^*:={v}(z_i^*,w_{\rm L}):=\frac{w_{\rm L}+z_i^*}2
\end{align*}
and

\begin{align}
	S(\rho,\rho_0):=\left\{\begin{array}{lll}
		\sqrt{\displaystyle{\frac{\rho(p(\rho)-p(\rho_0))}{\rho_0(\rho-\rho_0)}}}
		,\quad\mbox{if}\;\rho\ne\rho_0,\\
		\sqrt{p'(\rho_0)},\quad\mbox{if}\;\rho=\rho_0.
	\end{array}\right.
	\label{s(,)}
\end{align}

We call this approximated 1-rarefaction wave a {\it 1-rarefaction fan}.

\vspace*{10pt}
{\it Step 2}.\\
In this step, we replace the above constant states  with functions of $x$ and $t$ as follows:

In view of \eqref{transformation}, we construct ${u}^{\varDelta}_1(x,t)$.

We first determine the approximation of $\tilde{z},\tilde{w}$ in \eqref{transformation} 
as follows.
\begin{align*}
	\begin{alignedat}{2}
		\tilde{z}^{\varDelta}_1
		=&z_{\rm L}-
\int^{x_{j-1}}_{x_0}
		\zeta(u^{\varDelta}_{n,0}(x))dx,\;
		\tilde{w}^{\varDelta}_1=w_{\rm L}-
 \int^{x_{j-1}}_{x_0}
		\zeta(u^{\varDelta}_{n,0}(x))dx.
	\end{alignedat}
	\end{align*}
We set
\begin{align}
	\begin{alignedat}{2}
		&\check{z}^{\varDelta}_1(x,t)=&&\tilde{z}^{\varDelta}_1
		+\int^{x_{j-1}}_{x_0}
		\zeta(u^{\varDelta}_{n,0}(x))dx
        +\int^x_{x^{\varDelta}_1}\zeta(u_{\rm L})dy
+\left\{g_1(x,t;u_{\rm L})+V(u_{\rm L})\right\}(t-t_n)
		,\\
		&\check{w}^{\varDelta}_1(x,t)=&&\tilde{w}^{\varDelta}_1
		+ \int^{x_{j-1}}_{x_0}
		\zeta(u^{\varDelta}_{n,0}(x))dx+\int^x_{x^{\varDelta}_1}\zeta(u_{\rm L})dy
+\left\{g_2(x,t;u_{\rm L})+V(u_{\rm L})\right\}
		(t-t_n),
	\end{alignedat}\label{appro1-2}
\end{align}
where $x^{\varDelta}_1=x_{j-1}$, 
\begin{align}
V(u)=q_{\ast}(u)-\alpha m,
\label{V}
\end{align}
$q_{\ast}(u)$ is the flux of $\eta_{\ast}(u)$ defined by
\begin{align*}
q_{\ast}(u)=m\left(\frac12\frac{m^2}{\rho^2}+\frac{\rho^{\gamma-1}}{\gamma-1}\right) 
\end{align*}and $u^{\varDelta}_{n,0}(x)$
is a piecewise constant function defined by
\begin{align}
u^{\varDelta}_{n,0}(x)=
\begin{cases}
u^n_j,\quad &x\in [x_{j-1},x_{j+1})\quad (j\in J_n,\;\text{$n$ is even})	,\vspace*{0.5ex}\\
u^n_j,\quad &x\in [x_{j-1},x_{j+1})\quad (j\in J_n,\;j\ne0,2N_x,\;\;\text{$n$ is odd})	,\vspace*{0.5ex}\\
u^n_0,\quad &x\in [0,x_1)\quad (j=0,\;\;\text{$n$ is odd}),\vspace*{0.5ex}\\
u^n_{2N_x},&x\in [x_{2N_x-1},x_{2N_x})\quad (j=2N_x,\;\;\text{$n$ is odd}).	
\end{cases}
\label{def-u0}
\end{align}

Using $\check{u}^{\varDelta}_1(x,t)$, we next define ${u}^{\varDelta}_1(x,t)$ as follows. 
\begin{align}
	\begin{alignedat}{2}
		&{z}^{\varDelta}_1(x,t)=&&\tilde{z}^{\varDelta}_1
+\int^{x_{j-1}}_{x_0}
		\zeta(u^{\varDelta}_{n,0}(x))dx+\int^x_{x^{\varDelta}_1}
\zeta(\check{u}^{\varDelta}_1(y,t))dy\\&&&	+
\left\{g_1(x,t;\check{u}^{\varDelta}_1)+V(u_{\rm L})\right\}(t-t_n)
		,\\
		&{w}^{\varDelta}_1(x,t)=&&\tilde{w}^{\varDelta}_1
		+\int^{x_{j-1}}_{x_0}
		\zeta(u^{\varDelta}_{n,0}(x))dx+\int^x_{x^{\varDelta}_1}
\zeta(\check{u}^{\varDelta}_1(y,t))dy\\&&&
		+\left\{g_2(x,t;\check{u}^{\varDelta}_1)+V(u_{\rm L})\right\}(t-t_n).
	\end{alignedat}\label{appro1}
\end{align}

\begin{remark}${}$
\begin{enumerate}
\item We notice that approximate solutions ${z}^{\varDelta}_1,{w}^{\varDelta}_1$ 
and $\tilde{z}^{\varDelta}_1,\tilde{w}^{\varDelta}_1$ correspond to 
	$z,w$ and $\tilde{z},\tilde{w}$ in \eqref{transformation}, respectively.
\item 
For $t>t_n$, our approximate solutions will satisfy 
\begin{align}
	\begin{alignedat}{2} \int^{x_{j-1}}_{x_0}&
		\zeta(u^{\varDelta}(x,t_{n+1-}))dx+\int^{t_{n+1}}_{t_n}\sum_{0\leq x\leq x_{j-1}}(\sigma[\eta_{\ast}]-[q_{\ast}])dt\\
		&=\int^{x_{j-1}}_{x_0}
		\zeta(u^{\varDelta}_{n,0}(x))dx+V(u_{\rm L}){\varDelta}t
		+o({\varDelta}x).
	\end{alignedat}
	\label{mass-conservation}	
\end{align}
In \eqref{appro1}, we thus employ the right hand side of \eqref{mass-conservation} 
instead of the left hand side.
\item Our construction of approximate solutions uses the iteration method twice (see \eqref{appro1-2} and \eqref{appro1}) 
to deduce \eqref{iteration}. 
\end{enumerate}

\end{remark}

First, by the implicit function theorem, we determine a propagation speed $\sigma_2$ and $u_2=(\rho_2,m_2)$ such that 
\begin{itemize}
	\item[(1.a)] $z_2:=z(u_2)=z^*_2$
	\item[(1.b)] the speed $\sigma_2$, the left state ${u}^{\varDelta}_1(x_2,t_{n.5})$ and the right state $u_2$ satisfy the Rankine--Hugoniot conditions, i.e.,
	\begin{align*}
		f(u_2)-f({u}^{\varDelta}_1(x^{\varDelta}_2(t_{n.5}),t_{n.5}))=\sigma_2(u_2-{u}^{\varDelta}_1(x^{\varDelta}_2(t_{n.5}),t_{n.5})),
	\end{align*}
\end{itemize}
where $x^{\varDelta}_2(t)
=x_j+
\sigma_2(t-t_n)$. Then we fill up by ${u}^{\varDelta}_1(x)$ the sector where $t_n\leqq{t}<t_{n+1},x_{j-1}\leqq{x}<x^{\varDelta}_2(t)$ (see Figure \ref{case1-1cell}).

\begin{figure}[htbp]
	\begin{center}
		\hspace{-2ex}
		\includegraphics[scale=0.3]{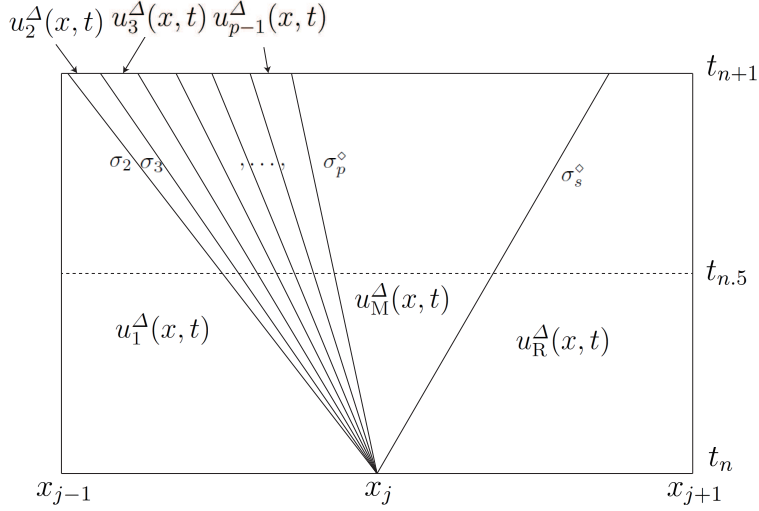}
	\end{center}
	\caption{The approximate solution in the case where a 1-rarefaction and 
		a 2-shock arise in the cell.}
	\label{case1-1cell}
\end{figure}

Assume that $u_k$, ${u}^{\varDelta}_k(x,t)$, a propagation speed $\sigma_k$ and $x^{\varDelta}_{k}(t)$  are defined. Then we similarly determine
$\sigma_{k+1}$ and $u_{k+1}=(\rho_{k+1},m_{k+1})$ such that 
\begin{itemize}
	\item[($k$.a)] $z_{k+1}:=z(u_{k+1})=z^*_{k+1}$,
	\item[($k$.b)] $\sigma_{k}<\sigma_{k+1}$,
	\item[($k$.c)] the speed 
	$\sigma_{k+1}$, 
	the left state ${u}^{\varDelta}_k(x^{\varDelta}_{k+1}(t_{n.5}),t_{n.5})$ and the right state $u_{k+1}$ satisfy 
	the Rankine--Hugoniot conditions, 
\end{itemize}
where $x^{\varDelta}_{k+1}(t)=x_j+\sigma_{k+1}(t-t_n)$. Then we fill up by ${u}^{\varDelta}_k(x,t)$ the sector where $t_n\leqq{t}<t_{n+1},x^{\varDelta}_{k}(t)\leqq{x}<x^{\varDelta}_{k+1}(t)$ (see Figure \ref{case1-1cell}).

We construct ${u}^{\varDelta}_{k+1}(x,t)$ as follows.

We first determine
\begin{align*}
	\begin{alignedat}{2}
		&\tilde{z}^{\varDelta}_{k+1}=&&z_{k+1}-\int^{x_{j-1}}_{x_0}
		\zeta(u^{\varDelta}_{n,0}(x))dx-V(u_{\rm L})\frac{{\varDelta}t}{2}
 -\sum^{k}_{l=1}\int^{x^{\varDelta}_{l+1}(t_{n.5})}_{x^{\varDelta}_l(t_{n.5})}
		\zeta(u^{\varDelta}_l(x,t_{n.5}))dx,
\end{alignedat}
\end{align*}
\begin{align*}
	\begin{alignedat}{2}
&\tilde{w}^{\varDelta}_{k+1}=&&w_{k+1}-\int^{x_{j-1}}_{x_0}
\zeta(u^{\varDelta}_{n,0}(x))dx-V(u_{\rm L})\frac{{\varDelta}t}{2}-\sum^{k}_{l=1}\int^{x^{\varDelta}_{l+1}(t_{n.5})}_{x^{\varDelta}_l(t_{n.5})}
		\zeta(u^{\varDelta}_l(x,t_{n.5}))dx,
	\end{alignedat}
\end{align*}
where $x^{\varDelta}_1(t)=x_{j-1},\;x^{\varDelta}_l(t)=x_j+\sigma_l(t-t_n)\quad
(l=2,3,\ldots,k+1)$ and $t_{n.5}$ is defined in \eqref{terminology}.

We next define $\check{u}^{\varDelta}_{k+1}$ as follows.
\begin{align*}
	\begin{alignedat}{2}
		&\check{z}^{\varDelta}_{k+1}(x,t)=&&\tilde{z}^{\varDelta}_{k+1}+\int^{x_{j-1}}_{x_0}
		\zeta(u^{\varDelta}_{n,0}(x))dx+V(u_{\rm L})(t-t_n)
		+\sum^{k}_{l=1}
		\int^{x^{\varDelta}_{l+1}(t)}_{x^{\varDelta}_l(t)}
		\zeta(u^{\varDelta}_l(x,t))dx\\&&&
		+\int^x_{x^{\varDelta}_{k+1}(t)}
		\zeta(u_{k+1})dy
		+g_1(x,t;u_{k+1})(t-t_{n.5})
		,\end{alignedat}
	\end{align*}\begin{align*}
	\begin{alignedat}{2}
		&\check{w}^{\varDelta}_{k+1}(x,t)=&&\tilde{w}^{\varDelta}_{k+1}
		+ \int^{x_{j-1}}_{x_0}
		\zeta(u^{\varDelta}_{n,0}(x))dx+V(u_{\rm L})(t-t_n)
		+\sum^{k}_{l=1}
		\int^{x^{\varDelta}_{l+1}(t)}_{x^{\varDelta}_l(t)}
		\zeta(u^{\varDelta}_l(x,t))dx\\&&&
		+\int^x_{x^{\varDelta}_{k+1}(t)}
		\zeta(u_{k+1})dy
		+g_2(x,t;u_{k+1})(t-t_{n.5}).
	\end{alignedat}
\end{align*}

Finally, using $\check{u}^{\varDelta}_{k+1}(x,t)$, we define ${u}^{\varDelta}_{k+1}(x,t)$ as follows.

\begin{align}
	\begin{alignedat}{2}
		&{z}^{\varDelta}_{k+1}(x,t)&=&\tilde{z}^{\varDelta}_{k+1}+ \int^{x_{j-1}}_{x_0}
		\zeta(u^{\varDelta}_{n,0}(x))dx+V(u_{\rm L})(t-t_n)
		+\sum^{k}_{l=1}
		\int^{x^{\varDelta}_{l+1}(t)}_{x^{\varDelta}_l(t)}
		\zeta(u^{\varDelta}_l(x,t))dx\\&&&
		+\int^x_{x^{\varDelta}_{k+1}(t)}
		\zeta(\check{u}^{\varDelta}_{k+1}(y,t))dy
		+g_1(x,t;\check{u}^{\varDelta}_{k+1})(t-t_{n.5})
		,\\
		&{w}^{\varDelta}_{k+1}(x,t)&=&\tilde{w}^{\varDelta}_{k+1}
		+ \int^{x_{j-1}}_{x_0}
		\zeta(u^{\varDelta}_{n,0}(x))dx+V(u_{\rm L})(t-t_n)
		+\sum^{k}_{l=1}
		\int^{x^{\varDelta}_{l+1}(t)}_{x^{\varDelta}_l(t)}
		\zeta(u^{\varDelta}_l(x,t))dx\\&&&+\int^x_{x^{\varDelta}_{k+1}(t)}\zeta(
		\check{u}^{\varDelta}_{k+1}(y,t))dy
		+g_2(x,t;\check{u}^{\varDelta}_{k+1})(t-t_{n.5}).
	\end{alignedat}
\label{appr-k}
\end{align}

By induction, we define $u_i$, ${u}^{\varDelta}_i(x,t)$ and $\sigma_i$ $(i=1,\ldots,p-1)$.
Finally, we determine a propagation speed $\sigma_p$ and $u_p=(\rho_p,m_p)$ such that
\begin{itemize}
	\item[($p$.a)] $z_p:=z(u_p)=z^*_p$,
	\item[($p$.b)] the speed $\sigma_p$, 
	and the left state ${u}^{\varDelta}_{p-1}(x^{\varDelta}_{p}(t_{n.5}),t_{n.5})$ and the right state $u_p$ satisfy the Rankine--Hugoniot conditions, 
\end{itemize}where $x^{\varDelta}_{p}(t)=x_j+\sigma_{p}(t-t_n)$. 
We then fill up by ${u}^{\varDelta}_{p-1}(x,t)$ and $u_p$ the sector where
$t_n\leqq{t}<t_{n+1},x^{\varDelta}_{p-1}(t)
\leqq{x}<x^{\varDelta}_{p}(t)$ 
and the line $t_n\leqq{t}<t_{n+1},x=x^{\varDelta}_{p}(t)$, respectively.

Given $u_{\rm L}$ and $z_{\rm M}$ with $z_{\rm L}\leqq{z}_{\rm M}$, we denote 
this piecewise functions of $x$ and $t$ 1-rarefaction wave by 
$R_1^{\varDelta}(u_{\rm L},z_{\rm M},x,t)$.

On the other hand, we construct ${u}^{\varDelta}_{\rm R}(x,t)$ as follows. 

We first set 
\begin{align*}
	\begin{alignedat}{2}
		\tilde{z}^{\varDelta}_{\rm R}
				=z_{\rm R}- \int^{x_{j+1}}_{x_0}
		\zeta(u^{\varDelta}_{n,0}(x))dx,\;
		\tilde{w}^{\varDelta}_{\rm R}
		=w_{\rm R}- \int^{x_{j+1}}_{x_0}
		\zeta(u^{\varDelta}_{n,0}(x))dx.
	\end{alignedat}
\end{align*}
We next construct $\check{u}^{\varDelta}_{\rm R}$
\begin{align*}
\begin{alignedat}{2}
	\check{z}^{\varDelta}_{\rm R}(x,t)&&=&\tilde{z}^{\varDelta}_{\rm R}
+ \int^{x_{j+1}}_{x_0}
		\zeta(u^{\varDelta}_{n,0}(x))dx+V(u_{\rm R})(t-t_n)
	+\int^x_{x_{j+1}}\zeta(u_{\rm R})dy\\&&&	+g_1(x,t;u_{\rm R})(t-t_n)
	,\\
	\check{w}^{\varDelta}_{\rm R}(x,t)&&=&\tilde{w}^{\varDelta}_{\rm R}
	+\int^{x_{j+1}}_{x_0}
		\zeta(u^{\varDelta}_{n,0}(x))dx+V(u_{\rm R})(t-t_n)+\int^x_{x_{j+1}}\zeta(u_{\rm R})dy
	\\&&&+g_2(x,t;u_{\rm R})(t-t_n).
\end{alignedat}
\end{align*}

Using $\check{u}^{\varDelta}_{\rm R}(x,t)$, we define ${u}^{\varDelta}_{\rm R}(x,t)$ as follows. \begin{align}
	\begin{alignedat}{2}
		{z}^{\varDelta}_{\rm R}(x,t)&=&&\tilde{z}^{\varDelta}_{\rm R}
+ \int^{x_{j+1}}_{x_0}
		\zeta(u^{\varDelta}_{n,0}(x))dx+V(u_{\rm R})(t-t_n)+\int^x_{x_{j+1}}\zeta(\check{u}_{\rm R}(y,t))dy\\&&&
		+g_1(x,t;\check{u}_{\rm R})(t-t_n)
		,\\
		{w}^{\varDelta}_{\rm R}(x,t)&=&&\tilde{w}^{\varDelta}_{\rm R}
		+ \int^{x_{j+1}}_{x_0}
		\zeta(u^{\varDelta}_{n,0}(x))dx+V(u_{\rm R})(t-t_n)+\int^x_{x_{j+1}}\zeta(\check{u}_{\rm R}(y,t))dy\\&&&
		+g_2(x,t;\check{u}_{\rm R})(t-t_n).
	\end{alignedat}
\label{appr-R}
\end{align}

Now we fix ${u}^{\varDelta}_{\rm R}(x,t)$ and ${u}^{\varDelta}_{p-1}(x,t)$. 
Let $\sigma_s$ be 
the propagation speed of the 2-shock connecting $u_{\rm M}$ and $u_{\rm R}$.
Choosing ${\sigma}^{\diamond}_p$ near to $\sigma_p$, ${\sigma}^{\diamond}_s$ 
near to 
$\sigma_s$ and $u^{\diamond}_{\rm M}$ near to $u_{\rm M}$, we fill up by ${u}^{\varDelta}_{\rm M}(x,t)$ the gap between $x=x_j+{\sigma}^{\diamond}_{p}
(t-{t}_n)$ and $x=x_j+{\sigma}^{\diamond}_s(t-{t}_n)$, such that 
\begin{itemize}
	\item[(M.a)] $\sigma_{p-1}<\sigma^{\diamond}_p<\sigma^{\diamond}_s$, 
	\item[(M.b)] the speed ${\sigma}^{\diamond}_p$, the left and right states 
	${u}^{\varDelta}_{p-1}(x^{\diamond}_{p},t_{n.5}),{u}^{\varDelta}_{\rm M}(x^{\diamond}_{p},t_{n.5})$ 
	satisfy the Rankine--Hugoniot conditions,
	\item[(M.c)] the speed ${\sigma}^{\diamond}_s$, the left and right 
	states ${u}^{\varDelta}_{\rm M}(x^{\diamond}_{s},t_{n.5}),{u}^{\varDelta}_{\rm R}(x^{\diamond}_{s},t_{n.5})$ satisfy the Rankine--Hugoniot conditions, 
\end{itemize}
where $x^{\diamond}_{p}:=x_j+\sigma^{\diamond}_{p}{\varDelta}
/2$, $x^{\diamond}_s:=x_j+\sigma^{\diamond}_s{\varDelta}
/2$ and ${u}^{\varDelta}_{\rm M}(x,t)$ defined as follows.

We first set
\begin{align*}
	\begin{alignedat}{2}
		\tilde{z}^{\varDelta}_{\rm M}&=&&z^{\diamond}_{\rm M}
		- \int^{x_{j+1}}_{x_0}
		\zeta(u^{\varDelta}_{n,0}\left(x\right))dx- V(u_{\rm R})
		\frac{{\varDelta}t}{2}-\int^{x^{\varDelta}_{\rm R}(t_{n.5})}_{x_{j+1}}
	\zeta(u^{\varDelta}_{\rm R}(x,t_{n.5}))dx,\\	\tilde{w}^{\varDelta}_{\rm M}&=&&w^{\diamond}_{\rm M}-\int^{x_{j+1}}_{x_0}
		\zeta(u^{\varDelta}_{n,0}\left(x\right))dx- V(u_{\rm R})
		\frac{{\varDelta}t}{2}-\int^{x^{\varDelta}_{\rm R}(t_{n.5})}_{x_{j+1}}
		\zeta(u^{\varDelta}_{\rm R}(x,t_{n.5}))dx,
	\end{alignedat}
\end{align*}
where $x^{\varDelta}_{\rm R}(t)=j{\varDelta}x+\sigma_{\rm R}(t-t_n)$.

We construct $\check{u}^{\varDelta}_{\rm M}$
\begin{align*}
	\begin{alignedat}{2}
		\check{z}^{\varDelta}_{\rm M}(x,t)&=&&\tilde{z}^{\varDelta}_{\rm M}
+ \int^{x_{j+1}}_{x_0}
		\zeta(u^{\varDelta}_{n,0}(x))dx+V(u_{\rm R})(t-t_n)+\int^{x^{\varDelta}_{\rm R}(t)}_{x_{j+1}}\zeta(u^{\varDelta}_{\rm R}(x,t))dy
	\\&&&	+\int_{x^{\varDelta}_{\rm R}(t)}^{x}\zeta(u_{\rm M})dy
		+g_1(x,t;u_{\rm M})(t-t_{n.5})
		,\\
		\check{w}^{\varDelta}_{\rm M}(x,t)&=&&\tilde{w}^{\varDelta}_{\rm M}
		+ \int^{x_{j+1}}_{x_0}
		\zeta(u^{\varDelta}_{n,0}(x))dx+V(u_{\rm R})(t-t_n)+\int^{x^{\varDelta}_{\rm R}(t)}_{x_{j+1}}\zeta(u^{\varDelta}_{\rm R}(x,t))dy\\&&&
		+\int_{x^{\varDelta}_{\rm R}(t)}^{x}\zeta(u_{\rm M})dy	+g_2(x,t;u_{\rm M})(t-t_{n.5}).
	\end{alignedat}
\end{align*}

Using $\check{u}^{\varDelta}_{\rm M}(x,t)$, we next define ${u}^{\varDelta}_{\rm M}(x,t)$ as follows. \begin{align}
	\begin{alignedat}{2}
		{z}^{\varDelta}_{\rm M}(x,t)&=&&\tilde{z}^{\varDelta}_{\rm M}
+ \int^{x_{j+1}}_{x_0}
		\zeta(u^{\varDelta}_{n,0}(x))dx+V(u_{\rm R})(t-t_n)+\int^{x^{\varDelta}_{\rm R}(t)}_{x_{j+1}}\zeta(u^{\varDelta}_{\rm R}(x,t))dy\\&&&	
		+\int_{x^{\varDelta}_{\rm R}(t)}^{x}\zeta(\check{u}^{\varDelta}_{\rm M}(y,t)_{\rm M})dy+g_1(x,t;\check{u}^{\varDelta}_{\rm M})(t-t_{n.5})
		,
		\\	{w}^{\varDelta}_{\rm M}(x,t)&=&&\tilde{w}^{\varDelta}_{\rm M}
		+\int^{x_{j+1}}_{x_0}
		\zeta(u^{\varDelta}_{n,0}(x))dx+V(u_{\rm R})(t-t_n)+\int^{x^{\varDelta}_{\rm R}(t)}_{x_{j+1}}\zeta(u^{\varDelta}_{\rm R}(x,t))dy\\&&&	
	+\int_{x^{\varDelta}_{\rm R}(t)}^{x}\zeta(\check{u}^{\varDelta}_{\rm M}(y,t)_{\rm M})dy+g_2(x,t;\check{u}^{\varDelta}_{\rm M})(t-t_{n.5}).
	\end{alignedat}
\label{appr-M}
\end{align}

We denote this approximate Riemann solution, which consists of \eqref{appr-k}, \eqref{appr-R}, \eqref{appr-R}
, by ${u}^{\varDelta}(x,t)$. The validity of the above construction is demonstrated in \cite[Appendix A]{T1}.

\begin{remark}\label{rem:middle-time}\normalfont
	${u}^{\varDelta}(x,t)$ satisfies the Rankine--Hugoniot conditions
	at the middle time of the cell, $t=t_{n.5}$.
\end{remark}


\begin{remark}\label{rem:approximate}\normalfont
	The approximate solution $u^{\varDelta}(x,t)$ is piecewise smooth in each of the 
	divided parts of the cell. Then, in the divided part, $u^{\varDelta}(x,t)$ satisfies
	\begin{align*}
		(u^{\varDelta})_t+f(u^{\varDelta})_x-g(x,u^{\varDelta})=O(\varDelta x).
	\end{align*}
\end{remark}

\section{The $L^{\infty}$ estimate of the approximate solutions}\label{sec:bound}
First aim in this section is to deduce from (\ref{remark2.1}) the following
theorem:
\begin{theorem}\label{thm:bound}
	For $x_{j-1}\leq x\leq x_{j+1}$,
	\begin{align}
		\begin{alignedat}{2}
			&\displaystyle {z}^{\varDelta}(x,t_{n+1-})&\geq&-M_{n+1}-L^n_j
			+\int^x_{x_0}\zeta({u}^{\varDelta}(y,t_{n+1-}))dy-{\it o}({\varDelta}x),\\
			&\displaystyle {w}^{\varDelta}(x,t_{n+1-})
			&\leq& M_{n+1}+L^n_j+\int^x_{x_0}\zeta({u}^{\varDelta}(y,t_{n+1-}))dy+\int^{t_{n+1}}_{t_n}\sum_{y<x_{j-1}}(\sigma[\eta_{\ast}]-[q_{\ast}])dt\\&&&+{\it o}({\varDelta}x),
		\end{alignedat}
		\label{goal}
	\end{align}
	where 
\begin{align}
M_{n+1}=M\left(1-\dfrac{{\varDelta}t}{4}\right)^{n+1},
\label{M-n}
\end{align}
$t_{n+1-}=(n+1){\varDelta}t-0$ and ${\it o}({\varDelta}x)$ depends only on $M$ in \eqref{char}. 
\end{theorem}

\vspace*{5pt}
Throughout this paper, by the Landau symbols such as $O({\varDelta}x)$,
$O(({\varDelta}x)^2)$ and $o({\varDelta}x)$,
we denote quantities whose moduli satisfy a uniform bound depending only on $M$ unless we specify them.  

Now, in the previous section, we have constructed 
$u^{\varDelta}(x,t)$ in {\bf Case 1}. When we consider $L^{\infty}$ estimates in this case, main difficulty is to obtain $(\ref{goal})_2$ along $R^{\varDelta}_1$. Therefore, we are concerned with $(\ref{goal})_2$ along $R^{\varDelta}_1$.

\subsection{Estimates of ${w}^{\varDelta}(x,t)$ along $R^{\varDelta}_1$ in Case 1}
In this step, we estimate ${w}^{\varDelta}(x,t)$ along $R^{\varDelta}_1$ in Case 1
of Section 2. We recall that ${u}^{\varDelta}$ along $R^{\varDelta}_1$ consists of ${u}^{\varDelta}_{k}\quad(k=1,2,3,\ldots,p-1)$.
In this case, ${w}^{\varDelta}(x,t)$ 
has the following properties, which is  proved in \cite[Appendix A]{T1}:
\begin{align}
{w}^{\varDelta}_{k+1}
	(x^{\varDelta}_{k+1}(t_{n.5}),t_{n.5})=&w_{k+1}
	={w}^{\varDelta}_k(x^{\varDelta}_{k+1}(t_{n.5}),t_{n.5})+{\it O}(({\varDelta}x)^{3\alpha-(\gamma-1)\beta})\nonumber
	\\&\hspace*{22ex}(k=1,\ldots,p-2),
	\label{w_i-w_{i+1}}
\end{align}
where $t_{n.5}$ is defined in \eqref{terminology}.

We first consider $\tilde{w}^{\varDelta}_1$. We recall that
\begin{align*}
	\begin{alignedat}{2}
		\tilde{w}^{\varDelta}_1=w_{\rm L}- \int^{x_{j-1}}_{x_0}
		\zeta(u^{\varDelta}_{n,0}(x))dx.
	\end{alignedat}
\end{align*}
From \eqref{remark2.1}, we have $\tilde{w}^{\varDelta}_1\leq M_n+L^n_j$.

Since 
\begin{align}
\check{u}^{\varDelta}_1(x,t)={u}^{\varDelta}_1(x,t)+
O(({\varDelta}x)^2),
\label{iteration}
\end{align}
recalling \eqref{mass-conservation}, we have 
\begin{align*}
	\begin{alignedat}{2}
		&{w}^{\varDelta}_1(x,t)&=&\tilde{w}^{\varDelta}_1
		+ \int^{x_{j-1}}_{x_0}
		\zeta(u^{\varDelta}_{n,0}(x))dx+V(u_{\rm L})(t-t_n)+\int^x_{x^{\varDelta}_1}
\zeta(\check{u}^{\varDelta}_1(y,t))dy\\&&&
		+g_2(x,t;\check{u}^{\varDelta})(t-t_{n})\\
		&&\leq&M_n+L^n_j+\int^{x_{j-1}}_{x_0}
		\zeta(u^{\varDelta}_{n,0}(x))dx+V(u_{\rm L})(t-t_n)+\int^x_{x^{\varDelta}_1}
		\zeta(\check{u}^{\varDelta}_1(y,t))dy\\&&&
		+g_2(x,t;{u}^{\varDelta})(t-t_{n})+o({\varDelta}x).
	\end{alignedat}
\end{align*}
If ${w}^{\varDelta}_1(x,t_{n+1-0})<M_n+L^n_j+I^n_j-\sqrt{{\varDelta}x}$, 
from \eqref{mass-conservation} and $M_{n+1}=M_n+O({\varDelta}x)$, we obtain  $\eqref{goal}_2$. Otherwise, 
from the argument \eqref{estimate1}--\eqref{estimate2}, 
regarding $M$ in \eqref{estimate1}--\eqref{estimate2} as $M_n+J^n_j$, 
we have $g_2(x,t;{u}^{\varDelta}_1)\leq-\frac 12(M_n+J^n_j)\leq-\frac 12M_n$. 
From \eqref{mass-conservation}, we conclude $\eqref{goal}_2$.

Next, we assume that \begin{align}
	\begin{alignedat}{2}
		&{w}^{\varDelta}_k(x,t)
		&\leq&M_n+L^n_j+\int^{x_{j-1}}_{x_0}
		\zeta(u^{\varDelta}_{n,0}(x))dx+V(u_{\rm L})(t-t_n)\\&&&+\int^x_{x_{j-1}}
		\zeta({u}^{\varDelta}(y,t))dy+o({\varDelta}x).
		\label{assumption}
	\end{alignedat}
\end{align}

We recall that\begin{align*}
	\begin{alignedat}{2}
		&\tilde{w}^{\varDelta}_{k+1}=&&w_{k+1}-\int^{x_{j-1}}_{x_0}
		\zeta(u^{\varDelta}_{n,0}(x))dx-V(u_{\rm L})\frac{{\varDelta}t}{2}-\sum^{k}_{l=1}\int^{x^{\varDelta}_{l+1}(t_{n.5})}_{x^{\varDelta}_l(t_{n.5})}
		\zeta(u^{\varDelta}_l(x,t_{n.5}))dx.
	\end{alignedat}
\end{align*}

From \eqref{w_i-w_{i+1}} and \eqref{assumption}, 
we have 
\begin{align*}
\tilde{w}^{\varDelta}_{k+1}\leq M_n+L^n_j+k\cdot{\it O}(({\varDelta}x)^{3\alpha-(\gamma-1)\beta})+o({\varDelta}x).	
\end{align*}

From a similar argument to ${w}^{\varDelta}_1$, we have
\begin{align*}
	\begin{alignedat}{2}
		&{w}^{\varDelta}_{k+1}(x,t)&=&\tilde{w}^{\varDelta}_{k+1}
		+\int^{x_{j-1}}_{x_0}
		\zeta(u^{\varDelta}_{n,0}(x))dx+V(u_{\rm L})(t-t_n)
		+\sum^{k}_{l=1}
		\int^{x^{\varDelta}_{l+1}(t)}_{x^{\varDelta}_l(t)}
		\zeta(u^{\varDelta}_l(x,t))dx\\&&&+\int^x_{x^{\varDelta}_{k+1}(t)}
		\zeta(\check{u}^{\varDelta}_{k+1}(y,t))dy
		+g_2(x,t;\check{u}^{\varDelta}_{k+1})(t-t_{n.5})\\
		&&\leq&M_n+L^n_j+\int^{x_{j-1}}_{x_0}
		\zeta(u^{\varDelta}_{n,0}(x))dx+V(u_{\rm L})(t-t_n)+\int^x_{x_{j-1}}\zeta(
		{u}^{\varDelta}(y,t))dy\\&&&+g_2(x,t;\check{u}^{\varDelta}_{k+1})(t-t_{n.5})+k\cdot{\it O}(({\varDelta}x)^{3\alpha-(\gamma-1)\beta})+o({\varDelta}x)\\&&&\quad
		(k=1,2,3,\ldots,p-1).
	\end{alignedat}
\end{align*}
From \eqref{order-p}, since $\left\{3\alpha-(\gamma-1)\beta\right\}p>1$, we conclude $\eqref{goal}_2$.

The remainder in this section is to prove the following theorem. This is important to ensures 
\eqref{remark2.1}.
\begin{theorem}
	\label{thm:average} 
	We assume that $u^{\varDelta}(x,t)$ satisfies \eqref{goal}.

	Then, if $E^{n+1}_j(\rho)\geq({\varDelta}x)^{\delta}$, it holds that
	\begin{align}
		\begin{split} 
			&-M_{n+1}-L^{n+1}_j+I^{n+1}_j
			-{\it o}({\varDelta}x)\leq{z}(E_{j}^{n+1}(u)),\\
			&w(E^{n+1}_j(u))\leq M_{n+1}+L^{n+1}_j+I^{n+1}_j+{\it o}({\varDelta}x),
		\end{split}
		\label{average}
	\end{align}
	where  $j\in J_{n+1}$ and 
	${\it o}({\varDelta}x)$ depends only on $M$  in \eqref{char}. 
\end{theorem}

{\it Proof of Theorem \ref{thm:average}.}
 For $x\in[x_{j-1},x_{j+1}]$, we set 
\begin{align*}
{z}^{\varDelta}_{\dagger}(x,t_{n+1-})=&{z}^{\varDelta}(x,t_{n+1-})-\int^x_{x_0}
\zeta\left({u}^{\varDelta}(y,t_{n+1-})\right)dy+\int^x_{x_{j-1}}
\eta_{\ast}\left({u}^{\varDelta}(y,t_{n+1-})\right)dy\\& 
-\int^x_{x_{j-1}}a^{n+1}_j{\rho}^{\varDelta}(y,t_{n+1-})dy
+\int^x_{x_{j-1}}Kdy,\\
{w}^{\varDelta}_{\dagger}(x,t_{n+1-})=&{w}^{\varDelta}(x,t_{n+1-})-\int^x_{x_0}
\zeta\left({u}^{\varDelta}(y,t_{n+1-})\right)dy+\int^x_{x_{j-1}}
\eta_{\ast}\left({u}^{\varDelta}(y,t_{n+1-})\right)dy \\&
-\int^x_{x_{j-1}}a^{n+1}_j{\rho}^{\varDelta}(y,t_{n+1-})dy
+\int^x_{x_{j-1}}Kdy,
\end{align*}
where $a^{n+1}_j=\dfrac{\partial\eta_{\ast}}{\partial\rho}(u^{n+1}_j)+
\dfrac{\partial\eta_{\ast}}{\partial m}(u^{n+1}_j)\left\{v^{n+1}_j-\left(\rho^n_j\right)^{\theta}\right\}$.

Then we notice that 
\begin{align*}
{\rho}^{\varDelta}_{\dagger}(x,t_{n+1-})=&{\rho}^{\varDelta}(x,t_{n+1-}),\\
{v}^{\varDelta}_{\dagger}(x,t_{n+1-})=&{v}^{\varDelta}(x,t_{n+1-})-\int^x_{x_0}
\zeta\left({u}^{\varDelta}(y,t_{n+1-})\right)dy+\int^x_{x_{j-1}}
\eta_{\ast}\left({u}^{\varDelta}(y,t_{n+1-})\right)dy \\&
-\int^x_{x_{j-1}}a^{n+1}_j{\rho}^{\varDelta}(y,t_{n+1-})dy
+\int^x_{x_{j-1}}Kdy.
\end{align*}

Since $L^n_j$ is positive, $\eqref{average}_2$ is more difficult than $\eqref{average}_1$. 
We thus treat with only $\eqref{average}_2$ in this proof.

\begin{align*}w(E^{n+1}_j(u))
=&\frac{\displaystyle \frac1{2{\varDelta}x}\int^{x_{j+1}}_{x_{j-1}}m^{\varDelta}(x,t_{n+1-})dx+\left(\frac1{2{\varDelta}x}\int^{x_{j+1}}_{x_{j-1}}{\rho}^{\varDelta}(x,t_{n+1-})dx\right)^{\theta}/\theta
}{\displaystyle \frac1{2{\varDelta}x}\int^{x_{j+1}}_{x_{j-1}}{\rho}^{\varDelta}(x,t_{n+1-})dx}
\\=&\frac{\displaystyle\frac1{2{\varDelta}x}\int^{x_{j+1}}_{x_{j-1}}{m}^{\varDelta}_{\dagger}(x,t_{n+1-})dx+\left(\frac1{2{\varDelta}x}\int^{x_{j+1}}_{x_{j-1}}{\rho}^{\varDelta}_{\dagger}(x,t_{n+1-})dx\right)^{\theta}/\theta
}{\displaystyle\frac1{2{\varDelta}x}\int^{x_{j+1}}_{x_{j-1}}{\rho}^{\varDelta}_{\dagger}(x,t_{n+1-})dx}\\
&+\frac{\displaystyle \frac1{2{\varDelta}x}\int^{x_{j+1}}_{x_{j-1}}{\rho}^{\varDelta}(x,t_{n+1-})\left\{	\int^{x_{j-1}}_{x_0}
	\eta_{\ast}\left({u}^{\varDelta}(y,t_{n+1-})\right)dy
\right\}dx
}{\displaystyle \frac1{2{\varDelta}x}\int^{x_{j+1}}_{x_{j-1}}{\rho}^{\varDelta}(x,t_{n+1-})dx}\\
&-\frac{\displaystyle \frac1{2{\varDelta}x}\int^{x_{j+1}}_{x_{j-1}}{\rho}^{\varDelta}(x,t_{n+1-})\left\{\left(\alpha-a^{n+1}_j\right)	\int^x_{x_0}
	{\rho}^{\varDelta}(y,t_{n+1-})dy                \right\}dx
}{\displaystyle \frac1{2{\varDelta}x}\int^{x_{j+1}}_{x_{j-1}}{\rho}^{\varDelta}(x,t_{n+1-})dx}
\\
&+\frac{\displaystyle \frac1{2{\varDelta}x}\int^{x_{j+1}}_{x_{j-1}}{\rho}^{\varDelta}(x,t_{n+1-})\left\{
	\int^{x_{j-1}}_{x_0}
	Kdy\right\}dx
}{\displaystyle \frac1{2{\varDelta}x}\int^{x_{j+1}}_{x_{j-1}}{\rho}^{\varDelta}(x,t_{n+1-})dx}
\\
=&A_1+A_2+A_3+A_4.	
\end{align*}

Considering $A_3$, we have
\begin{align*}
&\frac1{2{\varDelta}x}\int^{x_{j+1}}_{x_{j-1}}{\rho}^{\varDelta}(x,t_{n+1-})\left\{\left(\alpha-a^{n+1}_j\right)	\int^x_{x_0}
{\rho}^{\varDelta}(y,t_{n+1-})dy \right\}dx\\&=\frac1{2{\varDelta}x}\int^{x_{j+1}}_{x_{j-1}}{\rho}^{\varDelta}(x,t_{n+1-})dx\times\left(\alpha-a^{n+1}_j\right)	\int^{x_{j-1}}_{x_0}
{\rho}^{\varDelta}(y,t_{n+1-})dx \\
&\quad+\frac1{2{\varDelta}x}\int^{x_{j+1}}_{x_{j-1}}{\rho}^{\varDelta}(x,t_{n+1-})\left\{\left(\alpha-a^{n+1}_j\right)	\int^x_{x_{j-1}}
{\rho}^{\varDelta}(y,t_{n+1-})dy \right\}dx\\
&=A_{31}+A_{32}.
\end{align*}From the integration by parts, we have
\begin{align*}
A_{32}=&\frac1{2{\varDelta}x}\int^{x_{j+1}}_{x_{j-1}}{\rho}^{\varDelta}(x,t_{n+1-})dx\times\left(\alpha-a^{n+1}_j\right)	\int^{x_{j+1}}_{x_{j-1}}
{\rho}^{\varDelta}(y,t_{n+1-})dx\\
&-\frac1{2{\varDelta}x}\int^{x_{j+1}}_{x_{j-1}}\left\{\int^x_{x_{j-1}}
{\rho}^{\varDelta}(y,t_{n+1-})dy \right\}\left(\alpha-a^{n+1}_j\right) {\rho}^{\varDelta}(x,t_{n+1-})dx\\
=&\frac1{2{\varDelta}x}\int^{x_{j+1}}_{x_{j-1}}{\rho}^{\varDelta}(x,t_{n+1-})dx\times\left(\alpha-a^{n+1}_j\right)	\int^{x_{j+1}}_{x_{j-1}}
{\rho}^{\varDelta}(x,t_{n+1-})dx-A_{32}.
\end{align*}
We thus obtain 
\begin{align*}
	A_{32}=&\frac12\times
	\frac1{2{\varDelta}x}\int^{x_{j+1}}_{x_{j-1}}{\rho}^{\varDelta}(x,t_{n+1-})dx\times\left(\alpha-a^{n+1}_j\right)	\int^{x_{j+1}}_{x_{j-1}}
	{\rho}^{\varDelta}(y,t_{n+1-})dx.
\end{align*}

Therefore, we obtain 
\begin{align}
\begin{alignedat}{1}
w(E^{n+1}_j(u))
=&\frac{\displaystyle\frac1{2{\varDelta}x}\int^{x_{j+1}}_{x_{j-1}}{m}^{\varDelta}_{\dagger}(x,t_{n+1-})+\left(\frac1{2{\varDelta}x}\int^{x_{j+1}}_{x_{j-1}}{\rho}^{\varDelta}_{\dagger}(x,t_{n+1-})dx\right)^{\theta}/\theta
}{\displaystyle\frac1{2{\varDelta}x}\int^{x_{j+1}}_{x_{j-1}}{\rho}^{\varDelta}_{\dagger}(x,t_{n+1-})dx}\\
&+\int^{x_{j-1}}_{x_0}
\eta_{\ast}\left({u}^{\varDelta}_{n+1,0}(x)\right)dx
\\&-\alpha	\int^{x_{j-1}}_{x_0}
{\rho}^{\varDelta}(x,t_{n+1-})dx-\frac{\alpha-a^{n+1}_j}2	\int^{x_{j+1}}_{x_{j-1}}
{\rho}^{\varDelta}(x,t_{n+1-})dx\\&+Kx_{j-1}+\int^{x_{j-1}}_{x_{0}}\left\{\eta_{\ast}\left({u}^{\varDelta}(x,t_{n+1-})\right)-\eta_{\ast}\left({u}^{\varDelta}_{n+1,0}(x)\right)\right\}
dx.
\end{alignedat}\label{lemma 4.1}
	\end{align}

Here we introduce the following lemma. The proof is postponed to Appendix A. 
\begin{lemma}\label{lem:average1}
	If
	\begin{align}
		\begin{alignedat}{1}
			\frac1{2{\varDelta}x}\int_{x_{j-1}}
			^{x_{j+1}}{\rho}^{\varDelta}_{\dagger}(x,t_{n+1-0})
			dx\geq ({\varDelta}x)^{\delta}
		\end{alignedat}
		\label{assumption-average2}	
	\end{align}
and \begin{align}
	{w}^{\varDelta}_{\dagger}(x,t_{n+1-0})\leq& M_{n+1}+L^n_j+\int^x_{x_{j-1}}
	\eta_{\ast}\left({u}^{\varDelta}(y,t_{n+1-0})\right)dy-\int^x_{x_{j-1}}a^{n+1}_j{\rho}^{\varDelta}(y,t_{n+1-0})dy\nonumber 
	\\&+\int^x_{x_{j-1}}
	Kdy+\int^{t_{n+1}}_{t_n}\sum_{y<x_{j-1}}(\sigma[\eta_{\ast}]-[q_{\ast}])dt+o({\varDelta}x)\nonumber\\
	=&:
	A(x,t_{n+1-0})+o({\varDelta}x)\hspace*{2.0ex}(x\in [x_{j-1},x_{j+1}]),
	\label{def-A}
\end{align}
	the following holds
	\begin{align*}
		w(E_j^{n+1}({u}^{\varDelta}_{\dagger}))
		\leq 
		\bar{A}_j(t_{n+1-0})+o({\varDelta}x),
	\end{align*}
where $\displaystyle E_j^{n+1}({u}^{\varDelta}_{\dagger})=\frac1{2{\varDelta}x}\int^{x_{j+1}}_{x_{j-1}}
u^{\varDelta}_{\dagger}(x,t_{n+1-0})dx,\; \bar{A}_j(t_{n+1-0})=\frac1{2{\varDelta}x}\int^{x_{j+1}}_{x_{j-1}}A(x,t_{n+1-0})dx
$.
\end{lemma}

It follows from Theorem \ref{thm:bound} and this lemma that 
\begin{align}
\begin{alignedat}{1}
	w(E^{n+1}_j(u))
\leq&M_{n+1}+L^n_j+I^{n+1}_j+\int^{t_{n+1}}_{t_n}\sum_{y<x_{j-1}}(\sigma[\eta_{\ast}]-[q_{\ast}])dt
\\
&+\int^{x_{j-1}}_{x_{0}}\left\{\eta_{\ast}\left({u}^{\varDelta}(x,t_{n+1-})\right)-\eta_{\ast}\left({u}^{\varDelta}_{n+1,0}(x)\right)\right\}
dx\\
&+\frac1{2{\varDelta}x}\int^{x_{j+1}}_{x_{j-1}}
\int^x_{x_{j-1}}\left\{\eta_{\ast}\left({u}^{\varDelta}(y,t_{n+1-})\right)
-\eta_{\ast}\left({u}^{n+1}_j\right)\right\}dy	dx\\
&-\frac1{2{\varDelta}x}\int^{x_{j+1}}_{x_{j-1}}\int^x_{x_{j-1}}a^{n+1}_j\left({\rho}^{\varDelta}(y,t_{n+1-})-\rho^{n+1} _j \right)dydx
+o({\varDelta}x).
\end{alignedat}
\label{lemma 4.1 2}
\end{align}

To complete the proof of Theorem \ref{thm:average}, we must investigate
\begin{align*}
		\Gamma^{n+1}_j(y)=&\eta_{\ast}({u}^{\varDelta}(y,t_{n+1-}))-\eta_{\ast}(u^{n+1}_j)-a^{n+1}_j\left({\rho}^{\varDelta}(y,t_{n+1-})-\rho^{n+1}_j\right)
\end{align*}
in \eqref{lemma 4.1 2}, where 
\begin{align*}
 a^{n+1}_j=&\dfrac{\partial\eta_{\ast}}{\partial\rho}(u^{n+1}_j)+
 \dfrac{\partial\eta_{\ast}}{\partial m}(u^{n+1}_j)\left\{v^{n+1}_j-\left(\rho^{n+1}_j\right)^{\theta}\right\}.
\end{align*}
From the Taylor expansion, we have
	\begin{align}
		\begin{alignedat}{2}
			\eta_{\ast}\left({u}^{\varDelta}(y,t_{n+1-})\right)
			-
			\eta_{\ast}\left(u^{n+1}_j\right)=&
			\nabla\eta_{\ast}(u^{n+1}_j)\left({u}^{\varDelta}(y,t_{n+1-})-u^{n+1}_j\right)\\&
			+\int^1_0(1-\tau)\cdot{}^t\left({u}^{\varDelta}(y,t_{n+1-})-u^{n+1}_j\right)
			\\&\times\nabla^2\eta_{\ast}\left(u^{n+1}_j+\tau\left\{{u}^{\varDelta}(y,t_{n+1-})-u^{n+1}_j\right\}\right)d\tau\\&\times\left({u}^{\varDelta}(y,t_{n+1-})-u^{n+1}_j\right)\\
			=&\nabla\eta_{\ast}(u^{n+1}_j)\left({u}^{\varDelta}(y,t_{n+1-})-u^{n+1}_j\right)+R^{n+1}_j(y),
		\end{alignedat}
		\label{Taylor}	
	\end{align}where 
\begin{align*}
R^{n+1}_j(y)=&\int^1_0(1-\tau)\cdot{}^t\left({u}^{\varDelta}(y,t_{n+1-})-u^{n+1}_j\right)
	\nabla^2\eta_{\ast}\left(u^{n+1}_j+\tau\left\{{u}^{\varDelta}(y,t_{n+1-})-u^{n+1}_j\right\}\right)\\&\times\left({u}^{\varDelta}(y,t_{n+1-})-u^{n+1}_j\right)d\tau.
\end{align*}

We then deduce that
\begin{align*}
	\begin{alignedat}{2}
		\Gamma^{n+1}_j(y)=&
		\dfrac{\partial\eta_{\ast}}{\partial m}(u^{n+1}_j){\rho}^{\varDelta}(y,t_{n+1-})\left(w(y,t_{n+1-})-w^{n+1}_j\right)\\&-\dfrac{\partial\eta_{\ast}}{\partial m}(u^{n+1}_j)\int^1_0(1-\tau)(\theta+1)\left(\rho^{n+1}_j+\tau\left\{{\rho}^{\varDelta}(y,t_{n+1-})-\rho^{n+1}_j\right\}\right)^{\theta-1}d\tau\\
		&\times\left({\rho}^{\varDelta}(y,t_{n+1-})-\rho^{n+1}_j\right)^2+R^{n+1}_j(y)\\
		=&
		v^{n+1}_j{\rho}^{\varDelta}(y,t_{n+1-})\left(w(y,t_{n+1-})-w^{n+1}_j\right)\\&-v^{n+1}_j\int^1_0(1-\tau)(\theta+1)\left(\rho^{n+1}_j+\tau\left\{{\rho}^{\varDelta}(y,t_{n+1-})-\rho^{n+1}_j\right\}\right)^{\theta-1}d\tau\\
		&\times\left({\rho}^{\varDelta}(y,t_{n+1-})-\rho^{n+1}_j\right)^2+R^{n+1}_j(y).
	\end{alignedat}
\end{align*}

We thus obtain
\begin{align}
\begin{alignedat}{2}
\frac1{2{\varDelta}x}&\int^{x_{j+1}}_{x_{j-1}}\int^x_{x_{j-1}}
\Gamma^{n+1}_j(y,u^{n+1}_j)dydx
=\frac1{2{\varDelta}x}\int^{x_{j+1}}_{x_{j-1}}(x_{j+1}-x)\Gamma^{n+1}_j(x,u^{n+1}_j)dx\\
=&\frac1{2{\varDelta}x}\int^{x_{j+1}}_{x_{j-1}}(x_{j+1}-x)v^{n+1}_j{\rho}^{\varDelta}(y,t_{n+1-})\left(w(y,t_{n+1-})-w^{n+1}_j\right)dx\\
&-\frac1{2{\varDelta}x}\int^{x_{j+1}}_{x_{j-1}}(x_{j+1}-x)v^{n+1}_j\int^1_0(1-\tau)(\theta+1)\left(\rho^{n+1}_j+\tau\left\{{\rho}^{\varDelta}(y,t_{n+1-})-\rho^{n+1}_j\right\}\right)^{\theta-1}d\tau\\
&\times\left({\rho}^{\varDelta}(y,t_{n+1-})-\rho^{n+1}_j\right)^2dx
+\frac1{2{\varDelta}x}\int^{x_{j+1}}_{x_{j-1}}(x_{j+1}-x)R^{n+1}_j(x,u^{n+1}_j)dx\\
=&:B_1+B_2+B_3.
	\end{alignedat}
\label{lemma 3.4}
\end{align}

If $E^{n+1}_j(\rho)<({\varDelta}x)^{\delta}$, we find $B_1=o({\varDelta}x)$ and $B_2=o({\varDelta}x)$. 
Therefore, we devote to investigating the case where $E^{n+1}_j(\rho)\geq({\varDelta}x)^{\delta}$. 
From \eqref{def-u^n_j}, we recall that $E^{n+1}_j(z)=z^{n+1}_j,\;E^{n+1}_j(w)=w^{n+1}_j$.

We set 
\begin{align*}
	S=\left\{x\in[x_{j-1},x_{j+1}];w^{\varDelta}(x,t_{n-0})\leq M_{n+1}+L^{n+1}_j+I^{n+1}_j-
	\left({\varDelta}x\right)^{1/4}\right\}.
\end{align*}

If $\mu(S)/(2{\varDelta}x)\geq\left({\varDelta}x\right)^{1/4}$, from the 
Jensen inequality, we find that $w^{n+1}_j\leq M_{n+1}+L^{n+1}_j+I^{n+1}_j-
\left({\varDelta}x\right)^{1/2}/2$, where $\mu$ is the Lebesgue measure. In this case, since $B_1=O({\varDelta}x),
\;B_2=O({\varDelta}x),
\;B_3=O({\varDelta}x)$, we can obtain 
$\eqref{average}_2$.

Otherwise, we consider the following lemma.

\begin{lemma}\label{lem:average2}If $\mu(S)/(2{\varDelta}x)<\left({\varDelta}x\right)^{1/4}$, 
\begin{align*}
\frac1{2{\varDelta}x}\int^{x_{j+1}}_{x_{j-1}}\int^x_{x_{j-1}}
\Gamma^{n+1}_j(y)dydx\leq&\left(1+C_{\gamma}\alpha\int^1_0\rho_0(x)dx\right)\frac1{2{\varDelta}x}\int^{x_{j+1}}_{x_{j-1}}(x_{j+1}-x)R^{n+1}_j(x)dx\\
&+o({\varDelta}x).\end{align*}
\end{lemma}
\begin{proof}

We first treat with $B_1$ in \eqref{lemma 3.4}.
If $\mu(S)/(2{\varDelta}x)<\left({\varDelta}x\right)^{1/4}$, there exists a positive constant $C$ independent of
${\varDelta}x$ such that $w^{n+1}_j\geq M_{n+1}+L^{n+1}_j+I^{n+1}_j-C\left({\varDelta}x\right)^{1/4}$. We thus have
\begin{align*}
	\begin{alignedat}{2}
	|B_1|
		\leq&\frac{1}{2{\varDelta}x}\int^{x_{j+1}}_{x_{j-1}}\left|(x_{j+1}-x)v^{n+1}_j{\rho}^{\varDelta}(y,t_{n+1-})\left(w(y,t_{n+1-})-w^{n+1}_j\right)\right|dx\\
		\leq& O(1)\left\{\frac{1}{2{\varDelta}x}\int_{S}\left|w(y,t_{n+1-})-w^{n+1}_j
		\right|dx+\frac{1}{2{\varDelta}x}\int_{[x_{j-1},x_{j+1}]\setminus S}\left|w(y,t_{n+1-})-w^{n+1}_j
		\right|dx\right\}\\
		=&o({\varDelta}x).
	\end{alignedat}
\end{align*}

We next consider $B_2$.
Since $z^{n+1}_j\geq -M_{n+1}-L^{n+1}_j+I^{n+1}_j+O({\varDelta}x)$, we find $v^{n+1}_j\geq I^{n+1}_j$. If $v^{n+1}_j\geq0$, we have $B_2\leq0$. Therefore, we devotes to considering the case $v^{n+1}_j<0$. Since $I^{n+1}_j\leq v^{n+1}_j\leq 0$, from the conservation of mass, we have 
\begin{align*}
-v^{n+1}_j\leq -I^{n+1}_j\leq \alpha \int^1_0\rho_0(x)dx+o({\varDelta}x).
\end{align*}
On the other hand, we find that $(\rho^{n+1}_j)^{\theta}/\theta\geq M_{n+1}+L^{n+1}_j+I^{n+1}_j-C\left({\varDelta}x\right)^{1/4}\geq1/\theta$, choosing $M$ large enough.

If ${\rho}^{\varDelta}(y,t_{n+1-})\geq\rho^{n+1}_j/2$, we have
\begin{align*}
&\left(\rho^{n+1}_j+\tau\left\{{\rho}^{\varDelta}(y,t_{n+1-})-\rho^{n+1}_j\right\}\right)^{\theta-1}=
\dfrac{\left(\rho^{n+1}_j+\tau\left\{{\rho}^{\varDelta}(y,t_{n+1-})-\rho^{n+1}_j\right\}\right)^{\gamma-2}}{\left(\rho^{n+1}_j+\tau\left\{{\rho}^{\varDelta}(y,t_{n+1-})-\rho^{n+1}_j\right\}\right)^{\theta}}\\
&\qquad\leq\dfrac{\left(\rho^{n+1}_j+\tau\left\{{\rho}^{\varDelta}(y,t_{n+1-})-\rho^{n+1}_j\right\}\right)^{\gamma-2}}{\left(\rho^{n+1}_j+\tau\left\{{\rho}^{\varDelta}(y,t_{n+1-})-\rho^{n+1}_j\right\}\right)^{\theta}}
\leq\dfrac{\left(\rho^{n+1}_j+\tau\left\{{\rho}^{\varDelta}(y,t_{n+1-})-\rho^{n+1}_j\right\}\right)^{\gamma-2}}{\left(1/2\right)^{\theta}}.
\end{align*}
If ${\rho}^{\varDelta}(y,t_{n+1-})<\rho^{n+1}_j/2$, since\begin{align*}
	\begin{alignedat}{2}
		&\int^1_0(1-\tau)(\theta+1)\left(\rho^{n+1}_j+\tau\left\{{\rho}^{\varDelta}(y,t_{n+1-})-\rho^{n+1}_j\right\}\right)^{\theta-1}d\tau\left({\rho}^{\varDelta}(y,t_{n+1-})-\rho^{n+1}_j\right)^2\\
		&\quad=
	\dfrac{({\rho}^{\varDelta}(y,t_{n+1-}))^{\theta+1}}{\theta}-\dfrac{(\rho^{n+1}_j)^{\theta+1}}{\theta}	-\dfrac{(\theta+1)(\rho^{n+1}_j)^{\theta}}{\theta}\left({\rho}^{\varDelta}(y,t_{n+1-})-\rho^{n+1}_j\right)
	\leq (\rho^{n+1}_j)^{\theta+1}
	\end{alignedat}
\end{align*}
and
\begin{align*}
	\begin{alignedat}{2}
		&\int^1_0(1-\tau)\left(\rho^{n+1}_j+\tau\left\{{\rho}^{\varDelta}(y,t_{n+1-})-\rho^{n+1}_j\right\}\right)^{\gamma-2}d\tau\left({\rho}^{\varDelta}(y,t_{n+1-})-\rho^{n+1}_j\right)^2\\
		&\quad=
		\dfrac{({\rho}^{\varDelta}(y,t_{n+1-}))^{\gamma}}{\gamma(\gamma-1)}-
		\dfrac{(\rho^{n+1}_j)^{\gamma}}{\gamma(\gamma-1)}	-\dfrac{(\rho^{n+1}_j)^{\gamma-1}}{\gamma-1}\left({\rho}^{\varDelta}(y,t_{n+1-})-\rho^{n+1}_j\right)
\\
&\quad\geq\dfrac{\gamma-2+\left(\frac12\right)^{\gamma-1}}{2\gamma(\gamma-1)}(\rho^{n+1}_j)^{\gamma}
\geq\dfrac{\gamma-2+\left(\frac12\right)^{\gamma-1}}{2\gamma(\gamma-1)}(\rho^{n+1}_j)^{\theta+1}.
	\end{alignedat}
\end{align*}
For $C_{\gamma}$ in \eqref{CGamma}, we have\begin{align*}
	\begin{alignedat}{2}
		&\int^1_0(1-\tau)\left(\rho^{n+1}_j+\tau\left\{{\rho}^{\varDelta}(y,t_{n+1-})-\rho^{n+1}_j\right\}\right)^{\theta+1}d\tau\left({\rho}^{\varDelta}(y,t_{n+1-})-\rho^{n+1}_j\right)^2\\&\quad\leq C_{\gamma}
		\int^1_0(1-\tau)\left(\rho^{n+1}_j+\tau\left\{{\rho}^{\varDelta}(y,t_{n+1-})-\rho^{n+1}_j\right\}\right)^{\gamma-2}d\tau\left({\rho}^{\varDelta}(y,t_{n+1-})-\rho^{n+1}_j\right)^2\\&\quad\leq C_{\gamma}R^{n+1}_j(y).
	\end{alignedat}
\end{align*}
\end{proof}
From Lemma \ref{lem:average2}, we can complete the proof of Theorem \ref{thm:average}. 
\vspace*{0ex}

\section{Proof of Theorem \ref{thm:periodic}}
To deduce that the sum of $J^{n}_j$ is bounded, we prove the following lemma.
\begin{lemma}\label{lem:error}
	\begin{align}
		0\leq&\sum_{n\in{\bf N}_t}\int^1_0\left\{\eta_{\ast}\left({u}^{\varDelta}(x,t_{n-})\right)-\eta_{\ast}\left({u}^{\varDelta}_{n,0}(x)\right)\right\}dx\label{lemm 4.2 1}\\=&
		\sum_{\substack{j\in J_n\\n\in {\bf N}_t}}\int^{x_{j+1}}_{x_{j-1}}R^n_j(x)dx+o({\varDelta}x)\label{lemm 4.2 2}
		\\
		=&
		\int^1_0\left\{\eta_{\ast}\left({u}^{\varDelta}(x,t_{0-})\right)-\eta_{\ast}\left({u}^{\varDelta}_{2N_x,0}(x)\right)\right\}dx
		-\int^{1}_{0}\sum_{0\leq x\leq 1}(\sigma[\eta_{\ast}]-[q_{\ast}])dt+o({\varDelta}x)
		\nonumber
	\end{align}\begin{align}
		\leq&\int^1_0\eta_{\ast}(u_0(x))dx+o({\varDelta}x),\label{lemm 4.2 3}
		\\
		&\hspace*{-4ex}\left(1+C_{\gamma}\alpha\int^1_0\rho_0(x)dx\right)\sum_{\substack{j\in J_n\\n\in {\bf N}_t}}\frac1{2{\varDelta}x}\int^{x_{j+1}}_{x_{j-1}}(x_{j+1}-x)R^n_j(x)dx\leq K_{\gamma}\int^1_0\eta_{\ast}(u_0(x))dx.
		\label{lemm 4.2 4}
	\end{align}
\end{lemma}
\begin{proof}
	Recalling \eqref{def-u0}, we deduce from the Jensen inequality \eqref{lemm 4.2 1}; from \eqref{Taylor}, we obtain \eqref{lemm 4.2 2}; 
	we can find a similar argument to \eqref{lemm 4.2 3} in \cite[(6.12)]{T1};
	it follows from the second inequality that \eqref{lemm 4.2 4}. 
\end{proof}

Our approximate solutions satisfy the following propositions holds (these proofs are similar to \cite{T1}--\cite{T3}.).
\begin{proposition}\label{pro:compact}
The measure sequence
\begin{align*}
\eta_{\ast}(u^{\varDelta})_t+q(u^{\varDelta})_x
\end{align*}
lies in a compact subset of $H_{\rm loc}^{-1}(\Omega)$ for all weak entropy 
pair $(\eta_{\ast},q)$, where $\Omega\subset[0,1]\times[0,1]$ is any bounded
and open set. 
\end{proposition}
\begin{proposition} 
Assume that the approximate solutions $u^{\varDelta}$ are bounded and satisfy Proposition \ref{pro:compact}. Then there is a convergent subsequence $u^{\varDelta_n}(x,t)$
in the approximate solutions $u^{\varDelta}(x,t)$ such that
\begin{equation*}      
u^{\varDelta_n}(x,t)\rightarrow u(x,t)
\hspace{2ex}
\text{\rm a.e.,\quad as\;\;}n\rightarrow \infty.
\end{equation*} 
The function $u(x,t)$ is a global entropy solution
of the Cauchy problem \eqref{IP}.
\end{proposition}

\subsection{Existence of a time periodic solution}From Remark \ref{rem:approximate}, $u^{\varDelta}$ satisfy 
\begin{align*}
	(u^{\varDelta})_t+f(u^{\varDelta})_x-g(x,t,u^{\varDelta})=O(\varDelta x)
\end{align*}
on the divided part in the cell where $u^{\varDelta}$ are smooth. Moreover, $u^{\varDelta}$ satisfy an entropy condition (see \cite[Lemma 5.1--Lemma 5.4]{T1}) along 
discontinuous lines approximately. Then, applying the Green formula to $(u^{\varDelta})_t+f(u^{\varDelta})_x-g(x,t,u^{\varDelta})$ in the cell $x_{j-1}\leqq{x}<x_{j+1},\;{t}_n\leqq{t}<{t}_{n+1}\quad(j\in J_{n+1},\;n\in{\bf N}_t)$, we have
\begin{align}
	\begin{split}
		\rho^{n+1}_j=&\frac{\rho^n_{j+1}+\rho^n_{j-1}}2-\frac{\varDelta t}{2\varDelta x}\left\{m^n_{j+1}-m^n_{j-1}\right\}-R^n_{j+1}+R^n_{j-1}+o(\varDelta x), 
		\\
		m^{n+1}_j=&\frac{m^n_{j+1}+m^n_{j-1}}2-\frac{\varDelta t}{2\varDelta x}\left\{\frac{(m^n_{j+1})^2}{\rho^n_{j+1}}+p(\rho^n_{j+1})
		-\frac{(m^n_{j-1})^2}{\rho^n_{j-1}}-p(\rho^n_{j-1})\right\}\\
		&-S^n_{j+1}+S^n_{j-1}-{\varDelta t}\dfrac{\rho^n_{j+1}+\rho^n_{j-1}}{2}F(x_j,t_n)+o(\varDelta x),
	\end{split}
	\label{recurrence1}
\end{align}
where    
\begin{align}\begin{alignedat}{2}
	R^n_j=&\frac{(\varDelta t)^2}{8\varDelta x}\left\{
	\rho^n_j\left(H^n_j+G^n_j\right)+	
	\dfrac{m^n_j}{(\rho^n_j)^{\theta}}
	\left(H^n_j-G^n_j\right)\right\}
	,\\
	S^n_j=&
	\frac{\varDelta x}{4}(\rho^n_j)\zeta(u^n_j)
	+\frac{(\varDelta t)^2}{8\varDelta x}
	\biggl[2\rho^n_j\left\{H^n_j+G^n_j+2V(u^n_j)\right\}\\&\left.
+\dfrac{\rho^n_j(v^n_j)^2+(\rho^n_j)^{\gamma}}{(\rho^n_j)^{\theta}}\left(H^n_j-G^n_j\right)
-2m^n_j\right],
\quad(\text{recall \eqref{zeta} and \eqref{V}})
\end{alignedat}
\label{RS}
\end{align}
\begin{align*}
	\begin{alignedat}{2}
	G^n_j=&-K\lambda_1(u^n_j)+\dfrac{1}{\gamma(\gamma-1)}(\rho^n_j)^{\gamma+\theta}
		+\dfrac{1}{\gamma}(\rho^n_j)^{\gamma}v^n_j+\dfrac{1}{2}(\rho^n_j)^{\theta+1}(v^n_j)^2\\&-\alpha(\rho^n_j)^{\theta+1}
		+F(x_j,t_n)-\sum_{\substack{k\in J_n\\k+2\leq j}}F(x_{k+1},t_n)\xi^n_k,\\
    H^n_j=&-K\lambda_2(u^n_j)-\dfrac{1}{\gamma(\gamma-1)}(\rho^n_j)^{\gamma+\theta}
		-\dfrac{1}{\gamma}(\rho^n_j)^{\gamma}v^n_j-\dfrac{1}{2}(\rho^n_j)^{\theta+1}(v^n_j)^2\\&+\alpha(\rho^n_j)^{\theta+1}
		+F(x_j,t_n)-\sum_{\substack{k\in J_n\\k+2\leq j}}F(x_{k+1},t_n)\xi^n_k,
	\end{alignedat}
\end{align*}
where \begin{align*}
\xi^n_k=({m^n_{k+2}+m^n_{k}}){\varDelta x}-\frac{2\varDelta t}{3}\left\{\frac{(m^n_{k+2})^2}{\rho^n_{k+2}}+p(\rho^n_{k+2})
		-\frac{(m^n_{k})^2}{\rho^n_{k}}-p(\rho^n_{k})\right\}.
\end{align*}
Moreover, from \eqref{goal} and Theorem \ref{thm:average}, we have
\begin{align}
	-M^n_j-L^n_j+I^n_j-{\it o}({\varDelta}x)
	\leqq {z}(u^n_j),\quad
	{w}(u^n_j)
	\leqq M^n_j+L^n_j+I^n_j+{\it o}({\varDelta}x),\quad \rho^n_j\geqq0.
	\label{goal2}
\end{align}

Then, we define a sequence $\breve{u}^n_{j}=(\breve{\rho}^{n}_j,\breve{m}^{n}_j)$ as follows.

\begin{align}
&	\begin{alignedat}{2}
		\breve{\rho}^{n+1}_j=&\frac{\breve{\rho}^n_{j+1}+\breve{\rho}^n_{j-1}}2-\frac{\varDelta t}{2\varDelta x}\left\{\breve{m}^n_{j+1}-\breve{m}^n_{j-1}\right\}-\breve{R}^n_{j+1}+\breve{R}^n_{j-1},
		\\
		\breve{m}^{n+1}_j=&\frac{\breve{m}^n_{j+1}+\breve{m}^n_{j-1}}2-\frac{\varDelta t}{2\varDelta x}\left\{\frac{(\breve{m}^n_{j+1})^2}{\breve{\rho}^n_{j+1}}+p(\breve{\rho}^n_{j+1})
		-\frac{(\breve{m}^n_{j-1})^2}{\breve{\rho}^n_{j-1}}-p(\breve{\rho}^n_{j-1})\right\}\\
		&-\breve{S}^n_{j+1}+\breve{S}^n_{j-1}-{\varDelta t}\dfrac{\breve{\rho}^n_{j+1}+\breve{\rho}^n_{j-1}}{2}F(x_j,t_n),\quad(j\in J_{n+1},\;n\in {\bf N}_t)
	\end{alignedat}
	\label{recurrence2}\\
	&\breve{u}^0_{j}={u}^0_{j} \quad (j\in J_0),\nonumber
\end{align}
where $\breve{R}^n_{j},\;\breve{S}^n_{j}$ are defined by replacing ${u}^n_{j}$ with 
$\breve{u}^n_{j}$ in \eqref{recurrence1} respectively.

We notice that \eqref{recurrence2} is the recurrence relation obtained 
by removing $o(\varDelta x)$ from \eqref{recurrence1}.

Therefore, from \eqref{recurrence1}--\eqref{recurrence2}, there exists $\delta({\varDelta}x)>0$ satisfying $\delta({\varDelta}x)\rightarrow0$ as 
${\varDelta}x\rightarrow 0$, such that 
\begin{align}
	\begin{split}
		&-M^n_j-L^n_j+I^n_j-\delta({\varDelta}x)
		\leqq {z}(\breve{u}^n_j),\quad
		{w}(\breve{u}^n_j)
		\leqq M^n_j+L^n_j+I^n_j+\delta({\varDelta}x),\quad\breve{\rho}^n_j\geqq0.
	\end{split}
	\label{goal3}
\end{align}

Then, we define a map $F:{\bf R}^{4N_x+2}\rightarrow{\bf R}^{4N_x+2}$ as follows.
\begin{align}
\begin{alignedat}{2}
&\left(\left\{{z}(\breve{u}^0_{j})-I^0_j\right\}^{2N_x}_{j=0},
	\left\{{w}(\breve{u}^0_j)-I^0_j\right\}^{2N_x}_{j=0}\right)
	\\&\mapsto
	\left(\left\{{z}(\breve{u}^{2N_t}_{j})-I^{2N_t}_j\right\}^{2N_x}_{j=0},
	\left\{{w}(\breve{u}^{2N_t}_j)-I^{2N_t}_j\right\}^{2N_x}_{j=0}\right).
\end{alignedat}
	\label{map}
\end{align}
From \eqref{recurrence2}, $F$ is continuous.

To ensure that $u^0_j$ and $u^{2{\bf N}_t}_j$ are a same bounded set, we 
show the following. 
\begin{lemma}\label{lem:average3}
	\begin{align}
		\begin{split} 
			&-M-M/10+I^{2{\bf N}_t}_j\leq{z}(E_{j}^{2{\bf N}_t}(u)),\;
			w(E^{2{\bf N}_t}_j(u))\leq M+I^{2{\bf N}_t}_j+M/10.
		\end{split}
		\label{average2}
\end{align}\end{lemma}
\begin{proof}From Lemma \ref{lem:error}, there exists a constant $C$ independent of $M$ such that $\sum_{\substack{j\in J_n\\n\in {\bf N}_t}}J_{j}^{n}\le C$. On the other hand, 
	since $M_{2{\bf N}_t}$ in \eqref{average} satisfies $M_{2{\bf N}_t}$\linebreak$=M\left(1-\frac14{\varDelta}t\right)^{2{\bf N}_t}\rightarrow Me^{-\frac14}
	\quad({\varDelta}t\rightarrow0)$, it holds that $M_{2{\bf N}_t}<4/5M$, 
	choosing ${\varDelta}t$ small enough. Therefore, we deduce that $J_{j}^{2{\bf N}_t}+M_{2{\bf N}_t}\leq M+M/10$, choosing $M$ large enough. We can thus prove 
	this lemma.
\end{proof}
From Lemma \ref{lem:average3}, $F$ is the map from a bounded set to the same bounded set, choosing ${\varDelta}x$ small enough. Therefore, applying 
the Brouwer fixed point theorem to $F$, we have a fixed point 
\begin{align*}
\left(\left({z}(\breve{u}^0_{j})-I^0_j\right)^*,\left({w}(\breve{u}^0_j)-I^0_j\right)^*
	\right)=\left(\left({z}(\breve{u}^{2N_t}_{j})-I^{2N_t}_j\right)^*,\left({w}(\breve{u}^{2N_t}_j)-I^{2N_t}_j\right)^*
	\right). 
\end{align*}
This implies that 
\begin{align*}
\left(\left(\breve{\rho}^0_{j}\right)^*,\left(\breve{v}^0_j-I^0_j\right)^*
	\right)=\left(\left(\breve{\rho}^{2N_t}_{j}\right)^*,\left(\breve{v}^{2N_t}_j-I^{2N_t}_j\right)^*
	\right). 
\end{align*}

The remainder is to show $\left(\breve{v}^0_j\right)^*
	=\left(\breve{v}^{2N_t}_j\right)^*$ for any fixed ${\varDelta}x$. 
Assuming that there exists $j_*\in J_{2N_t}$ such that  $\left(\breve{v}^0_j\right)^*
	=\left(\breve{v}^{2N_t}_j\right)^*\quad (0\leq j<j_*)$ and $\left(\breve{v}^0_{j_*}\right)^*
	\ne\left(\breve{v}^{2N_t}_{j_*}\right)^*$, we deduce a contradiction.
\begin{align}
\left(\breve{v}^{2N_t}_{j_*}\right)^*-\left(\breve{v}^0_{j_*}\right)^*=&
{\left(\breve{v}^{2N_t}_{j_*}-I^{2N_t}_{j_*}\right)^*}+\left(I^{2N_t}_{j_*}\right)^*
-{\left(\breve{v}^{0}_{j_*}-I^{0}_{j_*}\right)^*}-\left(I^{0}_{j_*}\right)^*\nonumber\\
=&\left(I^{2N_t}_{j_*}-I^{0}_{j_*}\right)^*\nonumber\\
=&\int^{x_{j_*}}_{x_{j_*-1}}\left(\breve{\rho}^0_{j_*}\right)^*\left(\left(\breve{v}^{2N_t}_{j_*}\right)^*+\left(\breve{v}^0_{j_*}\right)^*\right)\left(\left(\breve{v}^{2N_t}_{j_*}\right)^*-\left(\breve{v}^0_{j_*}\right)^*\right)dx\nonumber\\
=&\left(\breve{\rho}^0_{j_*}\right)^*\left(\left(\breve{v}^{2N_t}_{j_*}\right)^*+\left(\breve{v}^0_{j_*}\right)^*\right)\int^{x_{j_*}}_{x_{j_*-1}}\left(\int^{x_{j_*}}_{x_{j_*-1}}
\left(I^{2N_t}_{j_*}-I^{0}_{j_*}\right)^*dx_1\right)
dx_0\nonumber
\\
=&\left\{\left(\breve{\rho}^0_{j_*}\right)^*\left(\left(\breve{v}^{2N_t}_{j_*}\right)^*+\left(\breve{v}^0_{j_*}\right)^*\right)\right\}^m\nonumber\\
&\times\int^{x_{j_*}}_{x_{j_*-1}}\left(\cdots\left(\int^{x_{j_*}}_{x_{j_*-1}}
\left(I^{2N_t}_{j_*}-I^{0}_{j_*}\right)^*
dx_m\right)\cdots\right)dx_0\nonumber\\
=&\left\{\left(\breve{\rho}^0_{j_*}\right)^*\left(\left(\breve{v}^{2N_t}_{j_*}\right)^*+\left(\breve{v}^0_{j_*}\right)^*\right){\varDelta}x\right\}^m
\left(I^{2N_t}_{j_*}-I^{0}_{j_*}\right)^*.
\label{recurrence3}
\end{align}

By choosing ${\varDelta}x$ small enough, we drive $\left\{\left(\breve{\rho}^0_{j}\right)^*\left(\left(\breve{v}^{2N_t}_{j_*}\right)^*+\left(\breve{v}^0_{j_*}\right)^*\right){\varDelta}x\right\}<1$. Since $m$ is arbitrary, 
this contradicts \eqref{recurrence3}.

Therefore, we obtain a fixed point 
\begin{align*}
\left(\left(\breve{\rho}^0_{j}\right)^*,\left(\breve{v}^0_j\right)^*
	\right)=\left(\left(\breve{\rho}^{2N_t}_{j}\right)^*,\left(\breve{v}^{2N_t}_j\right)^*
	\right). 
\end{align*}
Supplying the above as initial data, we can prove Theorem \ref{thm:periodic}.

\section{Open problem}
When we deduce \eqref{inhomo2}, we use the boundary condition $m|_{x=0}=0$. 
It should be noted that it is essential for our proof. Therefore, we cannot 
apply the present technique to the periodic boundary problem or other Dirichlet ones.

We have obtained the decay estimates \eqref{estimate1} and \eqref{estimate2} for 
large data. However, they do not still attain the convergence to an equilibrium $(\bar{\rho},0)$.

\appendix

\section{Proof of Lemma \ref{lem:average1}}
\begin{proof}
	
Due to space limitations, we denote $t_{n+1-0}$ by $T$ in this section.

	Set
	\begin{align*}
		\rho^{\varDelta}_{\dagger}(x,T)&:=\hat{\rho}(x,T)\left\{A(x,T)\right\}^{\frac{2}{\gamma-1}},\\
		m^{\varDelta}_{\dagger}(x,T)&:=\hat{m}(x,T)\left\{A(x,T)\right\}^{\frac{\gamma+1}{\gamma-1}},\\
		E_j^{n+1}(\rho^{\varDelta}_{\dagger})&:=\frac1{2{\varDelta}x}\int_{x_{j-1}}
		^{x_{j+1}}\hat{\rho}(x,T)\left\{A(x,T)\right\}^{\frac{2}{\gamma-1}}dx,
		\\
		E_j^{n+1}(m^{\varDelta}_{\dagger})&:=\frac1{2{\varDelta}x}\int_{x_{j-1}}
		^{x_{j+1}}\hat{m}(x,T)\left\{A(x,T)\right\}^{\frac{\gamma+1}{\gamma-1}}dx.
	\end{align*}
	Then, we find that  
	\begin{align}
		{w}(\hat{u}(x,T))\leq 1+{\it o}({\varDelta}x).
		\label{appendix2}
	\end{align}

	Let us  prove 
	\begin{align*}
		w(E_j^{n+1}(\rho^{\varDelta}_{\dagger}),E_j^{n+1}(m^{\varDelta}_{\dagger}))\leq \bar{A}_j(T)+o({\varDelta}x),
	\end{align*}
	where
		\begin{align*}
	\bar{A}_j(T)=\frac1{2{\varDelta}x}\int_{x_{j-1}}
		^{x_{j+1}}A(x,T)dx
	\end{align*}
	and
	\begin{align}
		w&(E_j^{n+1}(\rho^{\varDelta}_{\dagger}),E_j^{n+1}(m^{\varDelta}_{\dagger}))\nonumber\\
		&=E_j^{n+1}(m^{\varDelta}_{\dagger})/E_j^{n+1}(\rho^{\varDelta}_{\dagger})+\{E_j^{n+1}(\rho^{\varDelta}_{\dagger})\}^{\theta}
		/\theta\nonumber\\
		&=\frac{\displaystyle{\frac1{2{\varDelta}x}\int_{x_{j-1}}
				^{x_{j+1}}\hspace{-0.6ex}\hat{m}(x,T)\left\{A(x,T)\right\}^{\frac{\gamma+1}{\gamma-1}}dx}
			+{\left(\frac1{2{\varDelta}x}\int_{x_{j-1}}
				^{x_{j+1}}\hspace{-0.6ex}\hat{\rho}(x,T)
				\left\{A(x,T)\right\}^{\frac{2}{\gamma-1}}
				dx\right)^{\theta+1}}\hspace{-3ex}/{\theta}}
		{\displaystyle{\frac1{2{\varDelta}x}\int_{x_{j-1}}^{x_{j+1}}\hat{\rho}(x,T)	\left\{A(x,T)\right\}^{\frac{2}{\gamma-1}}dx}}.
\nonumber\\
		\label{lemma3.1-2}
	\end{align}

	{\it Step 1.}\\
	We find
	\begin{align*}
		E_j^{n+1}(\rho^{\varDelta}_{\dagger})&=\frac1{2{\varDelta}x}\int_{x_{j-1}}
		^{x_{j+1}}\hat{\rho}(x,T)\left\{A(x,T)\right\}^{\frac{\gamma+1}{\gamma-1}}\left\{A(x,T)\right\}^{-1}dx\\
		&=\left\{\bar{A}_j(T)\right\}^{-1}
		\frac1{2{\varDelta}x}\int_{x_{j-1}}
		^{x_{j+1}}\hat{\rho}(x,T)\left\{A(x,T)\right\}^{\frac{\gamma+1}{\gamma-1}}dx\\
		&\quad+\frac1{2{\varDelta}x}\int_{x_{j-1}}
		^{x_{j+1}}\hat{\rho}(x,T)\left\{A(x,T)\right\}^{\frac{\gamma+1}{\gamma-1}}
		\times\left(\left\{A(x,T)\right\}^{-1}-\left\{\bar{A}_j(T)\right\}^{-1}    \right)dx
		\\
		&=\left\{\bar{A}_j(T)\right\}^{-1}
		\frac1{2{\varDelta}x}\int_{x_{j-1}}
		^{x_{j+1}}\hat{\rho}(x,T)\left\{A(x,T)\right\}^{\frac{\gamma+1}{\gamma-1}}dx
		\\&\quad-\left\{\bar{A}_j(T)\right\}^{-1}\frac1{2{\varDelta}x}\int_{x_{j-1}}
		^{x_{j+1}}\hat{\rho}(x,T)\left\{A(x,T)\right\}^{\frac{2}{\gamma-1}}
		r(x,T)dx+o({\varDelta}x),
	\end{align*}
where $\displaystyle r(x,T)=A(x,T)-\bar{A}_j(T)$. Recalling \eqref{def-A}, we notice that 
	$\displaystyle r(x,T)=O({\varDelta}x)$.

	Substituting the above equation for (\ref{lemma3.1-2}), we obtain
	\begin{align}
		\lefteqn{w(E_j^{n+1}(\rho^{\varDelta}_{\dagger}),E_j^{n+1}(m^{\varDelta}_{\dagger}))}\nonumber\\
		=&\frac{\displaystyle{\frac1{2{\varDelta}x}\int_{x_{j-1}}
				^{x_{j+1}}\hspace{-0.6ex}\hat{m}(x,T)\left\{A(x,T)\right\}^{\frac{\gamma+1}{\gamma-1}}dx}
			+{\left(\frac1{2{\varDelta}x}\int_{x_{j-1}}
				^{x_{j+1}}\hspace{-0.6ex}\hat{\rho}(x,T)
				\left\{A(x,T)\right\}^{\frac{2}{\gamma-1}}
				dx\right)^{\theta+1}}\hspace{-3ex}/{\theta}}
		{\displaystyle{\left\{\bar{A}_j(T)\right\}^{-1}
				\frac1{2{\varDelta}x}\int_{x_{j-1}}
				^{x_{j+1}}\hat{\rho}(x,T)\left\{A(x,T)\right\}^{\frac{\gamma+1}{\gamma-1}}dx}
		}\nonumber\\
		&+\frac{\displaystyle{\frac1{2{\varDelta}x}\int_{x_{j-1}}
				^{x_{j+1}}\hspace{-0.6ex}\hat{m}(x,T)\left\{A(x,T)\right\}^{\frac{\gamma+1}{\gamma-1}}dx}
			+{\left(\frac1{2{\varDelta}x}\int_{x_{j-1}}
				^{x_{j+1}}\hspace{-0.6ex}\hat{\rho}(x,T)
				\left\{A(x,T)\right\}^{\frac{2}{\gamma-1}}
				dx\right)^{\theta+1}}\hspace{-3ex}/{\theta}}
		{\displaystyle{\left(
				\frac1{2{\varDelta}x}\int_{x_{j-1}}
				^{x_{j+1}}\hat{\rho}(x,T)\left\{A(x,T)\right\}^{\frac{2}{\gamma-1}}dx \right)^2
			}
		}
		\nonumber\\
		&\times \left\{\bar{A}_j(T)\right\}^{-1}\frac1{2{\varDelta}x}\int_{x_{j-1}}
		^{x_{j+1}}\hat{\rho}(x,T)\left\{A(x,T)\right\}^{\frac{2}{\gamma-1}}
		r(x,T)dx
		+o({\varDelta}x).
		\label{lemma3.1-5}
	\end{align}
	
	Set
	\begin{align}
		\lefteqn{\mu:=\frac{2}{\gamma+1}\frac1{\displaystyle{\left(
					\frac1{2{\varDelta}x}\int_{x_{j-1}}
					^{x_{j+1}}\hspace{-0.6ex}\hat{\rho}(x,T)
					\left\{A(x,T)\right\}^{\frac{2}{\gamma-1}}
					dx
					\right)^{\theta}}}}\nonumber\\
		&\times
		\frac{\displaystyle{\frac1{2{\varDelta}x}\int_{x_{j-1}}
				^{x_{j+1}}\hspace{-0.6ex}\hat{m}(x,T)\left\{A(x,T)\right\}^{\frac{\gamma+1}{\gamma-1}}dx}
			+{\left(\frac1{2{\varDelta}x}\int_{x_{j-1}}
				^{x_{j+1}}\hspace{-0.6ex}\hat{\rho}(x,T)
				\left\{A(x,T)\right\}^{\frac{2}{\gamma-1}}
				dx\right)^{\theta+1}}\hspace{-3ex}/{\theta}}
		{\displaystyle{
				\frac1{2{\varDelta}x}\int_{x_{j-1}}
				^{x_{j+1}}\hat{\rho}(x,T)\left\{A(x,T)\right\}^{\frac{2}{\gamma-1}}dx}
		}
		.
			\label{mu}
	\end{align}
	Then assume that the following holds.
	\begin{align}
		(E_j^{n+1}(\rho^{\varDelta}_{\dagger}))^{\theta+1}
		&\leq
		\frac1{2{\varDelta}x}\int_{x_{j-1}}
		^{x_{j+1}}(\hat{\rho}(x,T))^{\theta+1}
		\left\{A(x,T)\right\}^{\frac{\gamma+1}{\gamma-1}}dx\nonumber\\
		&\quad-\frac{\gamma+1}{2}\mu\left\{\bar{A}_j(T)\right\}^{-1}\left(
		\frac1{2{\varDelta}x}\int_{x_{j-1}}
		^{x_{j+1}}\hat{\rho}(x,T)\left\{A(x,T)\right\}^{\frac{2}{\gamma-1}}dx       \right)^{\theta}\nonumber\\
		&\quad\times\left(	\frac1{2{\varDelta}x}\int_{x_{j-1}}
		^{x_{j+1}}\hat{\rho}(x,T)
		\left\{A(x,T)\right\}^{\frac{2}{\gamma-1}}r(x,T)dx\right.\nonumber\\&\left.
		\quad-\frac1{2{\varDelta}x}\int_{x_{j-1}}
		^{x_{j+1}}\hat{\rho}(x,T)
		\left\{A(x,T)\right\}^{\frac{2}{\gamma-1}}dx\frac1{2{\varDelta}x}\int_{x_{j-1}}
		^{x_{j+1}}r(x,T)dx\right)
	\nonumber\\
		&\quad+o({\varDelta}x)\frac1{2{\varDelta}x}\int_{x_{j-1}}
		^{x_{j+1}}\hat{\rho}(x,T)
		\left\{A(x,T)\right\}^{\frac{2}{\gamma-1}}dx.\label{lemma3.1-13}
	\end{align}
	
	This estimate shall be proved in step 2--4.
	Then, substituting (\ref{lemma3.1-13}) for (\ref{lemma3.1-5}),
	we deduce from \eqref{appendix2} that
\begin{align*}
		w(E_j^{n+1}(\bar{\rho}),E_j^{n+1}(m^{\varDelta}_{\dagger}))
		\leq&\frac{\displaystyle{\frac1{2{\varDelta}x}\int_{x_{j-1}}
				^{x_{j+1}}\hspace{-1ex}\hat{\rho}(x,T)\left\{A(x,T)\right\}^{\frac{\gamma+1}{\gamma-1}}
				\hspace{-1ex}
				\left[\hat{v}(x,T)+\frac{\{\hat{\rho}(x,T)\}^{\theta}}{\theta}\right]dx}
		}{\displaystyle{\left\{\bar{A}_j(T)\right\}^{-1}\frac1{2{\varDelta}x}\int_{x_{j-1}}^{x_{j+1}}\hat{\rho}(x,T)\left\{A(x,T)\right\}^{\frac{\gamma+1}{\gamma-1}}dx}}\hspace{-0.5ex}
\\&+o({\varDelta}x)
		\\
		\leq&\bar{A}_j(T)+o({\varDelta}x).
	\end{align*}

	Therefore we must prove (\ref{lemma3.1-13}).
	Separating three steps, we derive this estimate.

	{\it Step 2.}\\
	From \eqref{assumption-average2},
we notice that 
	\begin{align*}
		|\mu|\leq{C}({\varDelta}x)^{-\theta\delta-\varepsilon},
	\end{align*}
	where $C$ depends only on $M$.

	In this step, we consider the first equation of (\ref{lemma3.1-5}):
	\begin{align*}{\left(\frac1{2{\varDelta}x}\int_{x_{j-1}}
			^{x_{j+1}}\hspace{-0.6ex}\hat{\rho}(x,T)
			\left\{A(x,T)\right\}^{\frac{2}{\gamma-1}}
			dx\right)^{\theta+1}}.
	\end{align*}
	Since $\theta\delta<1/2$, we first find 
	\begin{align*}
		E_j^{n+1}(\rho^{\varDelta}_{\dagger})
		=&\frac1{2{\varDelta}x}\int_{x_{j-1}}
		^{x_{j+1}}\hat{\rho}(x,T)\left\{A(x,T)\right\}^{\mu+\frac{2}{\gamma-1}}
		\left\{A(x,T)\right\}^{-\mu}dx\\
		=&\left\{\bar{A}_j(T)\right\}^{-\mu}
		\frac1{2{\varDelta}x}\int_{x_{j-1}}
		^{x_{j+1}}\hat{\rho}(x,T)\left\{A(x,T)\right\}^{\mu+\frac{2}{\gamma-1}}dx\\
		&-\mu
		\left\{\bar{A}_j(T)\right\}^{-\mu-1}\frac1{2{\varDelta}x}\int_{x_{j-1}}
		^{x_{j+1}}\hat{\rho}(x,T)
		\left\{A(x,T)\right\}^{\mu+\frac{2}{\gamma-1}}r(x,T)
		dx\\
		&+o({\varDelta}x)\frac1{2{\varDelta}x}\int_{x_{j-1}}
		^{x_{j+1}}\hat{\rho}(x,T)
		\left\{A(x,T)\right\}^{\frac{2}{\gamma-1}}dx\\
		:=&I_0-I_1+I_2.
	\end{align*}

	We next estimate $I_1$ as follows:
	\begin{align*}
		I_1&=\mu
		\left\{\bar{A}_j(T)\right\}^{-1}\frac1{2{\varDelta}x}\int_{x_{j-1}}
		^{x_{j+1}}\hat{\rho}(x,T)
		\left\{A(x,T)\right\}^{\frac{2}{\gamma-1}}r(x,T)dx
		\nonumber\\
		&\quad+o({\varDelta}x)\frac1{2{\varDelta}x}\int_{x_{j-1}}
		^{x_{j+1}}\hat{\rho}(x,T)\left\{A(x,T)\right\}^{\frac{2}{\gamma-1}}
		dx.
	\end{align*}
	Therefore, we have
	\begin{align*}
		\begin{split}
			E_j^{n+1}(\rho^{\varDelta}_{\dagger})
			=&\frac1{2{\varDelta}x}\int_{x_{j-1}}
			^{x_{j+1}}\hat{\rho}(x,T)\left\{A(x,T)\right\}^{\mu+\frac{2}{\gamma-1}}
			\left\{A(x,T)\right\}^{-\mu}dx\\
			=&\left\{\bar{A}_j(T)\right\}^{-\mu}
			\frac1{2{\varDelta}x}\int_{x_{j-1}}
			^{x_{j+1}}\hat{\rho}(x,T)\left\{A(x,T)\right\}^{\mu+\frac{2}{\gamma-1}}dx\\
			&-\mu
			\left\{\bar{A}_j(T)\right\}^{-1}\frac1{2{\varDelta}x}\int_{x_{j-1}}
			^{x_{j+1}}\hat{\rho}(x,T)
			\left\{A(x,T)\right\}^{\frac{2}{\gamma-1}}r(x,T)dx\\
			&+o({\varDelta}x)\frac1{2{\varDelta}x}\int_{x_{j-1}}
			^{x_{j+1}}\hat{\rho}(x,T)
			\left\{A(x,T)\right\}^{\frac{2}{\gamma-1}}dx.	
		\end{split}
	\end{align*}
	
	From the above, 
	we deduce that 
	\begin{align}
		(E_j^{n+1}(\rho^{\varDelta}_{\dagger}))^{\theta+1}
		&=\left(\left\{\bar{A}_j(T)\right\}^{-\mu}
		\frac1{2{\varDelta}x}\int_{x_{j-1}}
		^{x_{j+1}}\hat{\rho}(x,T)\left\{A(x,T)\right\}^{\mu+\frac{2}{\gamma-1}}dx\right.\nonumber\\
		&\quad\left.
		-\mu
		\left\{\bar{A}_j(T)\right\}^{-1}\frac1{2{\varDelta}x}\int_{x_{j-1}}
		^{x_{j+1}}\hat{\rho}(x,T)
		\left\{A(x,T)\right\}^{\frac{2}{\gamma-1}}r(x,T)dx
		\right)^{\theta+1}\nonumber\\
		&\quad+o({\varDelta}x)\frac1{2{\varDelta}x}\int_{x_{j-1}}
		^{x_{j+1}}\hat{\rho}(x,T)
		\left\{A(x,T)\right\}^{\frac{2}{\gamma-1}}dx
		\nonumber\\
		&=\left(
		\left\{\bar{A}_j(T)\right\}^{-\mu}
		\frac1{2{\varDelta}x}\int_{x_{j-1}}
		^{x_{j+1}}\hat{\rho}(x,T)\left\{A(x,T)\right\}^{\mu+\frac{2}{\gamma-1}}dx
		\right)
		^{\theta+1}\nonumber\\
		&\quad+(\theta+1)\left(
		\left\{\bar{A}_j(T)\right\}^{-\mu}
		\frac1{2{\varDelta}x}\int_{x_{j-1}}
		^{x_{j+1}}\hat{\rho}(x,T)\left\{A(x,T)\right\}^{\mu+\frac{2}{\gamma-1}}dx       \right)^{\theta}\nonumber\\
		&\quad\times-\mu
		\left\{\bar{A}_j(T)\right\}^{-1}\frac1{2{\varDelta}x}\int_{x_{j-1}}
		^{x_{j+1}}\hat{\rho}(x,T)
		\left\{A(x,T)\right\}^{\frac{2}{\gamma-1}}r(x,T)dx
		\nonumber\\
		&\quad+o({\varDelta}x)\frac1{2{\varDelta}x}\int_{x_{j-1}}
		^{x_{j+1}}\hat{\rho}(x,T)
		\left\{A(x,T)\right\}^{\frac{2}{\gamma-1}}dx
		\nonumber\\
		&=\left(
		\left\{\bar{A}_j(T)\right\}^{-\mu}
		\frac1{2{\varDelta}x}\int_{x_{j-1}}
		^{x_{j+1}}\hat{\rho}(x,T)\left\{A(x,T)\right\}^{\mu+\frac{2}{\gamma-1}}dx
		\right)
		^{\theta+1}\nonumber\\
		&\quad-\frac{\gamma+1}{2}\mu
		\left\{\bar{A}_j(T)\right\}^{-1}\left(
			\frac1{2{\varDelta}x}\int_{x_{j-1}}
		^{x_{j+1}}\hat{\rho}(x,T)\left\{A(x,T)\right\}^{\frac{2}{\gamma-1}}dx       \right)^{\theta}\nonumber\\
		&\quad\times\frac1{2{\varDelta}x}\int_{x_{j-1}}
		^{x_{j+1}}\hat{\rho}(x,T)
		\left\{A(x,T)\right\}^{\frac{2}{\gamma-1}}r(x,T)dx
		\nonumber\\
		&\quad+o({\varDelta}x)\frac1{2{\varDelta}x}\int_{x_{j-1}}
		^{x_{j+1}}\hat{\rho}(x,T)
		\left\{A(x,T)\right\}^{\frac{2}{\gamma-1}}dx
		.
		\label{lemma3.1-9}
	\end{align}

	{\it Step 3}\\
	Applying the Jensen inequality to the first term of the right-hand of (\ref{lemma3.1-9}), we have
	\begin{align}
		&\left(
		\left\{\bar{A}_j(T)\right\}^{-\mu}
		\frac1{2{\varDelta}x}\int_{x_{j-1}}
		^{x_{j+1}}\hat{\rho}(x,T)\left\{A(x,T)\right\}^{\mu+\frac{2}{\gamma-1}}dx
		\right)
		^{\theta+1}\nonumber
		\\
		&=\left(\frac{\displaystyle{	\left\{\bar{A}_j(T)\right\}^{-\mu}
				\frac1{2{\varDelta}x}\int_{x_{j-1}}
				^{x_{j+1}}\hat{\rho}(x,T)\left\{A(x,T)\right\}^
				{\mu+\frac{2}{\gamma-1}}
			dx}}
		{\displaystyle{\frac1{2{\varDelta}x}\int_{x_{j-1}}
				^{x_{j+1}}\left\{A(x,T)\right\}^{\frac{\gamma+1}{\gamma-1}\mu}dx}}\right)^{\theta+1}
		\nonumber\\
		&\quad\times\left(\frac1{2{\varDelta}x}\int_{x_{j-1}}
		^{x_{j+1}}\left\{A(x,T)\right\}^{\frac{\gamma+1}{\gamma-1}\mu}dx\right)^{\theta+1}
		\nonumber\\
		&=\left(\frac{\displaystyle{	\left\{\bar{A}_j(T)\right\}^{-\mu}
				\frac1{2{\varDelta}x}\int_{x_{j-1}}
				^{x_{j+1}}\hat{\rho}(x,T)\left\{A(x,T)\right\}^
				{\mu+\frac{2}{\gamma-1}}
				dx}}
		{\displaystyle{\frac1{2{\varDelta}x}\int_{x_{j-1}}
				^{x_{j+1}}\left\{A(x,T)\right\}^{\frac{\gamma+1}{\gamma-1}\mu}dx}}\right)^{\theta+1}
		\nonumber\\
		&\quad\times\left(\frac1{2{\varDelta}x}\int_{x_{j-1}}
		^{x_{j+1}}\left\{A(x,T)\right\}^{\frac{\gamma+1}{\gamma-1}\mu}dx\right)
	\nonumber\\
			&\quad\times\left(\left\{\bar{A}_j(T)\right\}^{\frac{\gamma+1}{2}\mu}+
			\frac{\gamma+1}{2}\mu\left\{\bar{A}_j(T)\right\}^{\frac{\gamma+1}{2}\mu-1}\frac1{2{\varDelta}x}\int_{x_{j-1}}
			^{x_{j+1}}r(x,T)dx+o({\varDelta}x)\right)
		\nonumber\\
&=\left(\frac{\displaystyle{	\left\{\bar{A}_j(T)\right\}^{-\mu}
				\frac1{2{\varDelta}x}\int_{x_{j-1}}
				^{x_{j+1}}\hat{\rho}(x,T)\left\{A(x,T)\right\}^
				{\mu-\frac{\gamma+1}{\gamma-1}\mu+\frac{2}{\gamma-1}}
				\left\{A(x,T)\right\}^{\frac{\gamma+1}{\gamma-1}\mu}dx}}
		{\displaystyle{\frac1{2{\varDelta}x}\int_{x_{j-1}}
				^{x_{j+1}}\left\{A(x,T)\right\}^{\frac{\gamma+1}{\gamma-1}\mu}dx}}\right)^{\theta+1}
		\nonumber\\
		&\quad\times\left(\frac1{2{\varDelta}x}\int_{x_{j-1}}
		^{x_{j+1}}\left\{A(x,T)\right\}^{\frac{\gamma+1}{\gamma-1}\mu}dx\right)\left\{\bar{A}_j(T)\right\}^{\frac{\gamma+1}{2}\mu}
		\nonumber\\
		&\quad
		+o({\varDelta}x){\displaystyle{	
				\frac1{2{\varDelta}x}\int_{x_{j-1}}
				^{x_{j+1}}\hat{\rho}(x,T)\left\{A(x,T)\right\}^
				{\frac{2}{\gamma-1}}
				dx}}
		\nonumber\\
		&\leq
		\frac1{2{\varDelta}x}\int_{x_{j-1}}
		^{x_{j+1}}(\hat{\rho}(x,T))^{\theta+1}
		\left\{A(x,T)\right\}^{\frac{\gamma+1}{\gamma-1}}dx\nonumber\\&\quad
		+o({\varDelta}x){\displaystyle{	
				\frac1{2{\varDelta}x}\int_{x_{j-1}}
				^{x_{j+1}}\hat{\rho}(x,T)\left\{A(x,T)\right\}^
				{\frac{2}{\gamma-1}}
				dx}}.
		\label{lemma3.1-10}
	\end{align}

	From (\ref{lemma3.1-9}) and (\ref{lemma3.1-10}), we obtain 
	(\ref{lemma3.1-13}) and complete the proof of lemma \ref{lem:average1}.
\end{proof}

\vspace*{-2.0ex}
\section{Construction and $L^{\infty}$ estimates of approximate solutions near the vacuum in Case 1}

In this step, we consider the case where $\rho_{\rm M}\leqq({\varDelta}x)^{\beta}$,
which means that $u_{\rm M}$ is near the vacuum. Since we cannot use the implicit 
function theorem, we must construct $u^{\varDelta}(x,t)$ in a different way.


\vspace*{5pt}
{\bf Case 1} A 1-rarefaction wave and a 2-shock arise.

In this case, we notice that $\rho_{\rm R}\leqq ({\varDelta}x)^{\beta},\;
z_{\rm R}\geqq -M_n-L^n_j+I^n_j$ and $w_{\rm R}\leqq  M_n+L^n_j+I^n_j$.
\vspace*{5pt}

\vspace*{5pt}
{\bf Case 1.1}
$\rho_{\rm L}>({\varDelta}x)^{\beta}$

We denote $u^{(1)}_{\rm L}$ a state satisfying $ w(u_{\rm L}^{(1)})=w(u_{\rm L})$ and 
$\rho^{(1)}_{\rm L}=({\varDelta}x)^{\beta}$. 
Let $u^{(2)}_{\rm L}$ be a state connected to ${u}^{\varDelta}_1(x_{j-1},t_{n+1-})$ on the right by 
$R_1^{\varDelta}(u_{\rm L},z^{(1)}_{\rm L},x,t_{n+1-})$. We set 
\begin{align*}
	(z^{(3)}_{\rm L},w^{(3)}_{\rm L})=
	\begin{cases}
		(z^{(2)}_{\rm L},w^{(2)}_{\rm L}),\quad\text{if $z^{(2)}_{\rm L}\geq D^n_j$},\\
		(D^n_j,w^{(2)}_{\rm L}),\quad\text{if $z^{(2)}_{\rm L}< D^n_j$},
	\end{cases}
\end{align*}
where 
\begin{align*}
D^n_j=&-M_{n+1}-L^{n}_j+\int^{x_{j-1}}_{x_0}\zeta({u}^{\varDelta}_{n,0}(x))dx+V(u_{\rm L}){\varDelta}t+
\int^{x_{j+1}}_{x_{j-1}}Kdx\\
&+
\int^{x_{j}+\lambda_1(u^{(2)}_{\rm L}){\varDelta}t}_{x_{j-1}}\eta(R_1^{\varDelta}(u_{\rm L},z^{(1)}_{\rm L},x,t_{n+1-}))dx.
\end{align*}



Then, we define ${u}^{\varDelta}(x,t)$ as follows.
\begin{align*}{u}^{\varDelta}(x,t)=
\begin{cases}
R_1^{\varDelta}(u_{\rm L},z^{(1)}_{\rm L},x,t),\hspace*{2ex}\text{if $x_{j-1}
	\leqq{x}\leqq x_{j}+\lambda_1(u^{(2)}_{\rm L})(t-{t}_{n})$}\\
\hspace*{19ex}\text{ and ${t}_{n}\leqq{t}<{t}_{n+1}$,}\vspace*{1ex}\\
u_{\rm Rw}(x,t),\hspace*{2ex}\text{if $ x_{j}+\lambda_1(u^{(2)}_{\rm L})(t-{t}_{n})$$<x
	\leqq x_{j}+\lambda_2(u_{\rm M},u_{\rm R})(t-{t}_{n})$}\\\hspace*{11ex}\text{ and ${t}_{n}\leqq{t}<{t}_{n+1}$,}\vspace*{1ex}\\
{u}^{\varDelta}_{\rm R}(x,t) \text{ defined in \eqref{appr-R}},\hspace*{2ex}\text{if $x_{j}+\lambda_2(u_{\rm M},u_{\rm R})(t-{t}_{n})$$<x
	\leqq x_{j+1}$ }\\\hspace*{28ex}\text{and ${t}_{n}\leqq{t}<{t}_{n+1}$,}
\end{cases}
\end{align*}
where (a) $\lambda_2(u_{\rm M},u_{\rm R})$ is a propagation speed of 2-shock wave; 
(b) $u_{\rm Rw}(x,t)$ is a\linebreak rarefaction wave connecting $u^{(3)}_{\rm L}$ and $u^{(4)}_{\rm L}$; (c) $u^{(4)}_{\rm L}$ is defined by $z^{(4)}_{\rm L}=\max\{z^{(3)}_{\rm L},z_{\rm M}\},$\linebreak$w^{(4)}_{\rm L}=w^{(3)}_{\rm L}$.

\begin{figure}[htbp]
	\begin{center}
		\vspace{-1ex}
		\hspace{2ex}
		\includegraphics[scale=0.3]{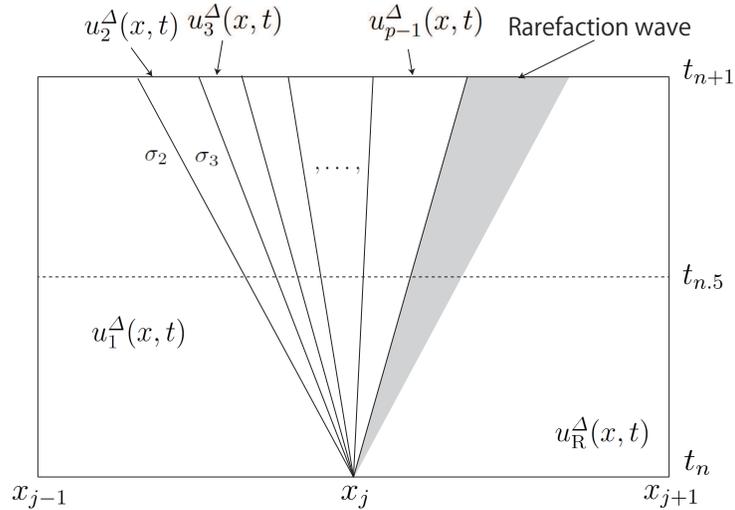}
	\end{center}\vspace*{-2ex}
	\caption{{\bf Case 1.1}: The approximate solution ${u}^{\varDelta}$ in the cell.}
	\label{Fig:case1.1}
\end{figure}

\vspace*{2ex}
{\bf Case 1.2} $\rho_{\rm L}\leqq({\varDelta}x)^{\beta}$

We set $(z^{(5)}_{\rm L},w^{(5)}_{\rm L})=(\max\{z_{\rm L},D^n_j\},
	\min\{w_{\rm L},U^n_j\})$,
where 
\begin{align*}
	U^n_j=&M_{n+1}+L^{n}_j+\int^{x_{j-1}}_{x_0}\zeta({u}^{\varDelta}_{n,0}(x))dx+V(u_{\rm L}){\varDelta}t.
\end{align*}

Then, we define ${u}^{\varDelta}(x,t)$ as follows.
\begin{align*}{u}^{\varDelta}(x,t)=
	\begin{cases}
		{u}^{\varDelta}_1(x,t)\text{ defined in \eqref{appro1}},\hspace*{2ex}\text{if $x_{j-1}
			\leqq{x}\leqq x_{j}+\lambda_1(u_{\rm L})(t-{t}_{n})$}\\
		\hspace*{28ex}\text{ and ${t}_{n}\leqq{t}<{t}_{n+1}$,}\vspace*{1ex}\\
		u_{\rm Rw}(x,t),\hspace*{2ex}\text{if $ x_{j}+\lambda_1(u_{\rm L})(t-{t}_{n})$$<x
			\leqq x_{j}+\lambda_2(u_{\rm M},u_{\rm R})(t-{t}_{n})$}\\\hspace*{11ex}\text{ and ${t}_{n}\leqq{t}<{t}_{n+1}$,}\vspace*{1ex}\\
		{u}^{\varDelta}_{\rm R}(x,t) \text{ defined in \eqref{appr-R}},\hspace*{2ex}\text{if $x_{j}+\lambda_2(u_{\rm M},u_{\rm R})(t-{t}_{n})$$<x
			\leqq x_{j+1}$ }\\\hspace*{28ex}\text{and ${t}_{n}\leqq{t}<{t}_{n+1}$,}
	\end{cases}
\end{align*}
where (a) $u_{\rm Rw}(x,t)$ is a rarefaction wave connecting $u^{(5)}_{\rm L}$ and $u^{(6)}_{\rm L}$; (b) $u^{(6)}_{\rm L}$ is defined by $z^{(6)}_{\rm L}=\max\{z^{(5)}_{\rm L},z_{\rm M}\},\;
w^{(6)}_{\rm L}=w^{(5)}_{\rm L}$.

\begin{remark} 
We notice that	
${\rho}^{\varDelta}(x,t)=O(({\varDelta}x)^{\beta})$ in (1.ii), (1.iii) and (2.i)--(2.iii). 
Therefore, the followings hold in these areas.

Although (1.ii) and (2.ii) are solutions of homogeneous isentropic gas dynamics (i.e., 
$g(x,t,u))=0$), they is also a solution of \eqref{IP} approximately
\begin{align*}
	(u^{\varDelta})_t+f(u^{\varDelta})_x-g(x,u^{\varDelta})=-g(x,u^{\varDelta})=O(({\varDelta}x)^{\beta}).
\end{align*}

In addition, discontinuities separating (1.i)--(1.iii) and (2.i)--(2.iii) satisfy \cite[Lemma 5.3]{T1}. 

\end{remark}

\subsection{$L^{\infty}$ estimates of approximate solutions}

We consider {\bf Case 1.1} in particular. It suffices to treat with $u_{\rm Rw}(x,t)$ 
in the region where $ x_{j}+\lambda_1(u^{(2)}_{\rm L})(t-{t}_{n})<x
\leqq x_{j}+\lambda_2(u_{\rm M},u_{\rm R})(t-{t}_{n})$ and ${t}_{n}\leqq{t}<{t}_{n+1}$. 
The other cases are similar to Theorem \ref{thm:bound}.

In this case, since ${\rho}^{\varDelta}(x,t)=O(({\varDelta}x)^{\beta})$, we have 
\begin{align}
\eta_{\ast}({u}^{\varDelta}(x,t))=O(({\varDelta}x)^{\beta}).
\label{vacuum-eta}
\end{align}
Moreover, we notice that \begin{align*}{w}^{\varDelta}(x,t_{n+1-})=w^{(2)}_{\rm L}=w(R_1^{\varDelta}(u_{\rm L},z^{(1)}_{\rm L},x_{j}+\lambda_1(u^{(2)}_{\rm L}){\varDelta}t,t_{n+1-})).\end{align*}
Applying Theorem \ref{thm:bound} to $R_1^{\varDelta}(u_{\rm L},z^{(1)}_{\rm L},x,t_{n+1-})$, we drive 
\begin{align*}
	\begin{alignedat}{2}
		&\displaystyle {w}^{\varDelta}(x,t_{n+1-})
		&\leq& M_{n+1}+L^n_j+\int^{x_{j}+\lambda_1(u^{(2)}_{\rm L}){\varDelta}t}_{x_0}\zeta({u}^{\varDelta}(y,t_{n+1-}))dy\\&&&+\int^{t_{n+1}}_{t_n}\sum_{y<x_{j-1}}(\sigma[\eta_{\ast}]-[q_{\ast}])dt+{\it o}({\varDelta}x)\\
		&&\leq& M_{n+1}+L^n_j+\int^{x}_{x_0}\zeta({u}^{\varDelta}(y,t_{n+1-}))dy+\int^{t_{n+1}}_{t_n}\sum_{y<x_{j-1}}(\sigma[\eta_{\ast}]-[q_{\ast}])dt\\&&&+{\it o}({\varDelta}x),
	\end{alignedat}
\end{align*}
which means $\eqref{goal}_2$.

Next, we notice that ${z}^{\varDelta}(x,t)\geq D^n_j$. In view of \eqref{mass-conservation} and \eqref{vacuum-eta}, we obtain $\eqref{goal}_1$.


\section*{Acknowledgements.}
N. Tsuge's research is partially supported by Grant-in-Aid for Scientific Research (C) 17K05315, Japan.

\end{document}